%% file: main.tex
\documentclass[11pt]{amsart}
\usepackage{geometry}                
\usepackage{graphicx}
\usepackage{amssymb}
\usepackage{epstopdf}

\usepackage[all]{xy}

\usepackage{amsthm}

\usepackage{todonotes}

\newcommand{\mainhcsref}{\ref{main_hcs}}	

\input{declarations}               

\title{On conjugacy separability of graph products of groups}
\author{Michal Ferov}
\address{Building 54, Mathematical Sciences, University of Southampton, Highfield, Southampton SO17 1BJ, UK}
\email{michalferov@gmail.com}


\date{}     

\begin{document}
 
\begin{abstract}
	We show that the class of $\C$-hereditarily conjugacy separable groups is closed under taking arbitrary graph products whenever the class $\C$ is an extension closed variety of finite groups. As a consequence we show that the class of $\C$-conjugacy separable groups is closed under taking arbitrary graph products. In particular, we show that right angled Coxeter groups are hereditarily conjugacy separable and 2-hereditarily conjugacy separable, and we show that infinitely generated right angled Artin groups are hereditarily conjugacy separable and $p$-hereditarily conjugacy separable for every prime number $p$.
\end{abstract}

\maketitle
\tableofcontents

\input{motivation}

\input{pro-C_topologies}

\input{graph_products}

\input{centraliser_conditions}

\input{special_amalgams}

\input{hcs}

\input{main_result_and_applications}


\end{document}

%% file: declarations.tex
\newtheorem{thrm}{\bf Theorem}[section]
\newtheorem{cor}[thrm]{\bf Corollary} 
\newtheorem{remark}[thrm]{\bf Remark} 
\newtheorem{lemma}[thrm]{\bf Lemma} 
\newtheorem{definition}[thrm]{\bf Definition}

\newcommand{\bigast}{\mathop{\scalebox{1.5}{\raisebox{-0.2ex}{$\ast$}}}}

\newcommand{\lnorm}{\mathop{\langle\langle}}
\newcommand{\rnorm}{\mathop{\rangle\rangle}}

\newcommand{\C}{\mathop{\mathcal{C}}}
\newcommand{\PT}{\mathop{\mathcal{PT}}}
\newcommand{\NC}{\mathop{\mathcal{N_C}}}
\newcommand{\BC}{\mathop{\mathcal{B_C}}}
\newcommand{\proC}{\text{pro-}\mathcal{C}}
\newcommand{\CCS}{\mathop{\mathcal{C}\text{-CS}}}
\newcommand{\CHCS}{\mathop{\mathcal{C}\text{-HCS}}}

\DeclareMathOperator{\specam}{\star}

\DeclareMathOperator{\normleq}{\unlhd}

\DeclareMathOperator{\Aut}{\mathop{Aut}}

\DeclareMathOperator{\End}{\mathop{End}}

\DeclareMathOperator{\sstar}{\mathop{star}}
\DeclareMathOperator{\link}{\mathop{link}}
\DeclareMathOperator{\FL}{\mathop{FL}}
\DeclareMathOperator{\LL}{\mathop{LL}}
\DeclareMathOperator{\supp}{\mathop{supp}}
\DeclareMathOperator{\PC}{\mathop{Pc}}

\DeclareMathOperator{\id}{\mathop{id}}

%% file: motivation.tex
\section{Introduction}

Let $G$ be a group. We say that $G$ is \emph{residually finite} (RF) if for every non-trivial element $g \in G$ there exists a a finite group $F$ and a group homomorphism $\phi \colon G \to F$ such that $\phi(g) \neq 1$ in $F$. We say that the group $G$ is \emph{conjugacy separable} (CS) if for every pair of elements $f,g \in G$ such that $f$ is not conjugate to $g$ in $G$ there is a finite group $F$ and a homomorphism $\phi \colon G \to F$ such that $\phi(f)$ is not conjugate to $\phi(g)$ in $F$.
	
	Group properties of this type are called \emph{separability properties}. In this paper we will be dealing with conjugacy separability and its various generalisations.
\subsection{Motivation}
 	 Separability properties provide an algebraic analogue to solvability of decision problems for finitely presented groups. Mal'cev \cite{malcev} proved that finitely presented residually finite groups have solvable word problem. Similarly, Mostowski \cite{mostowski} showed that finitely presented conjugacy separable groups have solvable conjugacy problem.

The following classes of groups are known to be conjugacy separable: virtually free groups (Dyer \cite{dyer}), virtually polycyclic groups (Formanek \cite{polycyclic_formanek}, Remeslennikov \cite{polycyclic_remeslennikov}), virtually surface groups (Martino \cite{armando}), limit groups (Chagas and Zalesskii \cite{limit}), finitely generated right angled Artin groups (Minasyan \cite{raags}), even Coxeter groups whose diagram does not contain $(4,4,2)$-triangles (Caprace and Minasyan \cite{racgs}), finitely presented residually free groups (Chagas and Zalesskii \cite{chagas}), one-relator groups with torsion (Minasyan and Zalesskii \cite{1-rel}) and fundamental groups of compact orientable 3-manifolds (Hamilton, Wilton and Zalesskii \cite{compact}).

Conjugacy separability is similar to residual finiteness but is much stronger. It can be easily seen that every CS group is RF, but the implication in the opposite direction does not hold. Perhaps the easiest example of a RF group which is not CS was given by Stebe \cite{stebe_sl3z} and independently by Remeslenikov \cite{remeslennikov} when they proved that $\text{SL}_3(\mathbb{Z})$ is not CS.

It is easy to see that being residually finite is a hereditary property: if a group $G$ is RF then every $H \leq G$ is residually finite as well. Unlike residual finiteness, conjugacy separability does not behave well with respect to subgroups, not even of finite index. Martino and Minasyan \cite{martino} showed that for every integer $m \geq 2$ there exists a finitely presented conjugacy separable group $T$ that contains a subgroup $S \leq T$ such that $|T:S| = m$ and $S$ is not CS. We say that a group $G$ is \emph{hereditarily conjugacy separable} (HCS) if $G$ is conjugacy separable and if $H \leq G$ is of finite index in $G$ then $H$ is CS as well.

Let $\mathcal{C}$ be a class of finite groups (we will always assume that classes of finite groups are closed under isomorphisms) and let $G$ be a group.	We say that $G$ is \emph{residually}-$\C$ if for every non-trivial $g \in G$ there is a group $F \in \C$ and a homomorphism $\phi \colon G \to F$ such that $\phi(g)$ is non-trivial in $F$. Similarly, we say that $G$ is \emph{$\C$-conjugacy separable} ($\CCS$) if for every tuple $f,g \in G$ such that $f$ is not conjugate to $g$ in $G$ there is a group $F \in \C$ and a homomorphism $\phi \colon G \to F$ such that $\phi(f)$ is not conjugate to $\phi(g)$ in $F$. We say that $G$ is \emph{$\C$-hereditarily conjugacy separable} ($\CHCS$) if it is $\CCS$ and every subgroup $H \leq G$, 
open in pro-$\C$ topology, is $\CCS$ ($H$ is open in pro-$\C$ topology if and only if there is $K \normleq G$ such that $K \leq H$ and $G/K \in \C$ - see Section \ref{pro_C}). If the class $\C$ satisfies certain closure properties we can equip the group $G$ with the so called pro-$\C$ topology and use basic terminology and methods from point-set topology to significantly simplify our proofs. Basic properties of pro-$\C$ topologies on groups, their connection to closure properties of the class $\C$ and definitions of residually-$\C$, $\CCS$ and $\CHCS$ groups in terms of pro-$\C$ topologies are given in Section \ref{pro_C}.

We say that a class of finite groups $\C$ is an \emph{extension closed variety of finite groups} if it is closed under taking subgroups, direct products, quotients and extensions. The most common examples of extension closed varieties of finite groups would be the class of all finite $p$-groups, where $p$ is a prime number, the class of all finite soluble groups or the class of all finite groups.  

In this paper we study the behaviour of $\C$-(hereditary) conjugacy separability under group constructions, where the class $\C$ is an extension closed variety of finite groups. It is easy to see that a direct product of $\CCS$ groups is again a conjugacy separable group, similarly for hereditary conjugacy separability. It was proved by Stebe \cite{stebe} and independently by Remeslennikov \cite{remeslennikov} that the class of CS groups is closed under taking free products and using this result one can show that a free product of HCS groups is again an HCS group. In his paper \cite{toinet} Toinet proved a specialised version of Dyer's theorem: free-by-(finite $p$) groups are $p\text{-CS}$. This result was generalised by Ribes and Zalesskii \cite{zalesskii}: finitely generated free-by-$\C$ groups are $\CCS$ whenever $\C$ is an extension closed variety of finite groups. Using the result of Ribes and Zalesskii one could easily generalise the result of Stebe and Remeslennikov to $\CCS$ and $\CHCS$ groups. Our aim is to show that the property of being $\C$-(H)CS is stable under graph products, group theoretic construction naturally generalising direct and free products in the category of groups.

\subsection{Statement of the results}
By a graph we will always mean a simplicial graph: i.e. graph $\Gamma$ is a tuple $(V\Gamma, E\Gamma)$, where $V\Gamma$ is a set and $E\Gamma \subseteq \binom{V\Gamma}{2}$. We call $V\Gamma$ the set of vertices of $\Gamma$ and $E\Gamma$ the set of edges of $\Gamma$.

Let $\Gamma$ be a graph and let $\mathcal{G}=\{G_v|v \in V\Gamma\}$ be a family of groups indexed by the vertices of $\Gamma$. The \emph{graph product} $\Gamma \mathcal{G}$ is the quotient of the free product $\bigast_{v \in V\Gamma}G_v$ obtained by adding all the relations of the form
	\begin{displaymath}
		g_u g_v = g_v g_u \mbox{ for all $g_u \in G_u, g_v \in G_v$ such that $\{u,v\} \in E\Gamma$}.
	\end{displaymath}
The groups $G_v$ will be usually referred to as \emph{vertex groups}.

Clearly if $\Gamma$ is a complete graph then $\Gamma \mathcal{G}$ is equal to the direct product $\prod_{v \in V\Gamma}G_v$ and if $\Gamma$ is the totally disconnected graph, meaning that $E\Gamma = \emptyset$, the resulting graph product is equal to the free product $\bigast_{v \in V\Gamma}G_v$. We say that the group $\Gamma \mathcal{G}$ is a \emph{finite graph product} if the corresponding graph $\Gamma$ is finite. 

If $G_v = \mathbb{Z}$, the additive group of integers, for all $v \in V\Gamma$, then we are talking about \emph{right angled Artin groups} (RAAGs), and if $G_v = C_2$, the cyclic group of order 2, we are talking about \emph{right angled Coxeter groups} (RACGs). In a way, RAAGs sit between free groups and free abelian groups. Since both free abelian groups and free groups are CS it is natural to ask whether RAAGs are CS as well. The positive answer to this question was given by Minasyan \cite{raags}, when he proved that finitely generated RAAGs are HCS. Toinet \cite{toinet} modified Minasyan's proof and showed that finitely generated RAAGs are $p\text{-HCS}$ for every prime number $p$. The main results of this paper are the following two theorems.
\begin{thrm}
	\label{main_cs}
	Assume that $\C$ is an extension closed variety of finite groups. Then the class of $\mathcal{C}\text{-CS}$ groups is closed under taking arbitrary graph products.
\end{thrm}
\begin{thrm}
	\label{main_chcs_infinite}
	Let $\C$ be an extension closed variety of finite groups. Then the class of $\CHCS$ groups is closed under taking arbitrary graph products.
\end{thrm}
Note that we do not impose any restrictions on the cardinality of the corresponding graph, i.e. $|V\Gamma|$ can be any cardinal.

\paragraph{\textbf{Acknowledgements}} The author would like to thank Ashot Minasyan for explaining his work in \cite{raags}, discussions, help and guidance.

%% file: pro-C_topologies.tex
\section{Pro-$\mathcal{C}$ topologies on groups}
\label{pro_C}
In this section we will explain basic properties of $\proC$ topologies on groups. In the profinite (or pro-$p$) case these are standard results and are part of mathematical folklore. We include this section in order to make this paper self-contained and readers familiar with $\proC$ topologies on groups can skip it.
 
What closure properties do we require the class $\C$ to have? We will be considering the following ones:
\begin{itemize}
	\item[(c1)]	subgroups: let $G \in \C$ and suppose that $H \leq G$; then $H \in \C$,
	\item[(c2)]	finite direct products: let $G_1, G_2 \in \C$; then $G_1 \times G_2 \in \C$,
	\item[(c3)] quotients: let $G \in \C$ and let $N \normleq G$; then $G/N \in \C$,
	\item[(c4)]	extensions: let $K, Q \in \C$ and let $G$ be a group such that the following sequence
		\begin{displaymath}
			1 \rightarrow K \rightarrow G \rightarrow Q \rightarrow 1
		\end{displaymath}
		is exact; then $G \in \C$.
\end{itemize}

Let $\C$ be a class of finite groups and let $G$ be a group. If $N \normleq G$ is such that $G/N \in \C$ then we say that $N$ is a \emph{co-$\C$} subgroup of $G$. We will use $\NC(G) = \{N \normleq G \mid G/N \in \C \}$ to denote the set of all co-$\C$ subgroups of $G$. We want the system of cosets $\BC(G) = \{gN \mid g \in G, N \in \NC(G)\}$ to form a basis of open sets for a topology on $G$, thus we need the set $\NC(G)$ to be closed under intersections. It can be easily seen that if $\C$ satisfies (c1) and (c2), then the set $\NC(G)$ is closed under intersections for every group $G$.

Suppose that the system of cosets $\BC(G)$ forms a basis of open sets for a topology on a group $G$. This topology is called the \emph{pro-$\C$ topology} on $G$ and we will use pro-$\C(G)$ to refer to this topology. If $\C$ is the class of all finite groups this topology is the profinite topology $\PT(G)$ and if $\C$ is the class of all finite $p$-groups, where $p$ is a prime number, this topology is referred to as pro-$p$ topology and is denoted by pro-$p(G)$.

We say that a subset $X \subseteq G$ is \emph{$\C$-separable} or \emph{$\C$-closed} in $G$ if the subset $X$ is closed pro-$\C(G)$. In other words, a subset $X \subseteq G$ is $\C$-separable if for every $g \in G \setminus X$ there is $N \in \NC(G)$ such that $gN \cap X = gN \cap XN = \emptyset$. Similarly we say that a subset $X \subseteq G$ is \emph{$\C$-open} if it is open in pro-$\C(G)$.
\subsection{Basic properties}
Unless stated otherwise we will only assume assume that the class $\C$ satisfies (c1) and (c2).

If the class $\C$ satisfies (c1) and (c2) then the $\text{pro-}\mathcal{C}$ topology on $G$ is well-defined for every group $G$. Note that it would be enough to assume that the class $\C$ is closed under subdirect products. However, if we assume that the class $\C$ is also closed under taking subgroups we see that equipping a group $G$ with $\proC$ topology is actually a faithfull functor from the category of groups to the category of topological groups: homomorphisms between groups are continuous maps with respect to the corresponding $\proC$ topologies and isomorphisms are homeomorphisms, thus preimages of $\C$-open/closed sets are $\C$-open/closed and isomorphic images of $\C$-open/closed sets are $\C$-open/closed. 

Obviously, the $\proC$ topology on a group $G$ is invariant under group translation: if $X \subseteq G$ is $\C$-closed in $G$ then so are the sets $gX$ and $Xg$ for all $g \in G$. 

The following lemma will help us to shorten and simplify proofs.
\begin{lemma}
\label{closed_simplification}
	Let $G$ be a group. Then $X \subseteq G$ is $\mathcal{C}$-closed in $G$ if and only if for every $g \in G \setminus X$ there is a group $F$ and a homomorphism $\phi: G \rightarrow F$, such that $\phi(g) \not\in \phi(X)$ and $\phi(X)$ is $\mathcal{C}$-closed in $F$.
\end{lemma}
\begin{proof}
	Suppose $X$ is $\mathcal{C}$-closed in $G$. Clearly if we take $F = G$ and $\phi = \id_G$ then $\phi(X) = X$ is $C$-closed in $G$ and $\phi(g) = g \not\in \phi(X) = X$ for all $g\in G \setminus X$.
	
	Let $X \subseteq G$ and suppose that for every $g \in G \setminus X$ there is a group homomorphism $\phi_g \colon G \to F_g$ such that $\phi_g(g) \not\in \phi_g(X)$ and $\phi_g(X)$ is $\mathcal{C}$-closed in $F_g$. We see that the set $\phi_g^{-1}(\phi_g(X))$ is $\C$-closed in $G$ as it is a preimage of a $\C$-closed set. Clearly $X = \bigcap_{g \in G\setminus X} \phi_g^{-1}(\phi_g(X))$ and thus $X$ is $\C$-closed in $G$ as it is intersection of $\C$-closed sets.  
\end{proof}

As we already said: a group $G$ is \emph{residually-$\mathcal{C}$} if for every $g \in G\setminus \{1\}$ there is a group $F \in \mathcal{C}$ and a group homomorphism $\pi \colon G \to F$ such that $\pi(g) \neq 1$ in $F$ or, equivalently, we say that $G$ is residually-$\mathcal{C}$ if for every $g \in G\setminus\{1\}$ there is $N \in \mathcal{N_C}(G)$ such that $g \not\in N$. Assuming that the class $\C$ satisfies (c1) and (c2) one can easily check that the following are equivalent:
	\begin{enumerate}
		\item	$G$ is residually-$\mathcal{C}$,
		\item $\{1\}$ is $\mathcal{C}$-closed in $G$,
		\item $\bigcap_{N \in \mathcal{N_C}(G)}N = \{1\}$,
		\item pro-$\mathcal{C}$(G) is Hausdorff.
	\end{enumerate}

Being residually-$\C$ is clearly a hereditary property.
\begin{remark}
	Let $G$ be a group and let $H \leq G$. If $G$ is residually-$\mathcal{C}$ then $H$ is residually-$\mathcal{C}$.
\end{remark}

Let $G$ be a group and assume that $H \leq G$. For an element $g \in G$ we will use $g^H$ to denote $\{hgh^{-1} \mid h \in H\} \subseteq G$, the set of $H$-conjugates of $H$. The symbol $\sim_H$ will denote the relation of being $H$-conjugates, i.e. $f \sim_H g$ if and only if $f \in g^H$. We can then restate the definition of $\C$-conjugacy separability in terms of pro-$\C$ topologies: group $G$ is $\CCS$ if the conjugacy class $g^G$ is $\C$-closed in $G$ for every $g \in G$.

Note that a specialised version of Lemma \ref{closed_simplification} is that a group $G$ is $\mathcal{C}\text{-CS}$ if and only if for every tuple of elements $f,g \in G$ such that $f \not\sim_G g$ there is a $\mathcal{C}\text{-CS}$ group $H$ and a homomorphism $\phi \colon G \to H$ such that $\phi(f) \not\sim_H \phi(g)$.

\subsection{$\C$-open and $\C$-closed subgroups}
Let $H \leq G$ and suppose that there is $N \in \NC(G)$ such that $N \leq H$. Then clearly $H$ is a union of cosets of $N$ and hence $H$ is $\C$-open in $G$ as it is a union of $\C$-open subsets of $G$. It was proved by Hall \cite[Theorem 3.1]{hall} that the opposite implication holds as well.
\begin{lemma}[Classification of $\mathcal{C}$-open subgroups]
	\label{open_subgroups}
	Let $G$ be a group and let $H \leq G$. Then $H$ is $\mathcal{C}$-open in $G$ if and only if there is $N \in \NC(G)$ such that $N \leq H$. Moreover, every $\mathcal{C}$-open subgroup is $\mathcal{C}$-closed in $G$ and is of finite index in $G$.
\end{lemma}

Lemma \ref{open_subgroups} allows us to restate the definition of $\C$-hereditary conjugacy separability in terms of pro-$\C$ topology: a group $G$ is $\CHCS$ if $G$ is $\CCS$ and $H \leq G$ is $\CCS$ as well whenever $H$ is $\C$-open in $G$.

Obviously, an intersection of $\C$-open subgroups is a $\C$-closed subgroup. It was proved by Hall \cite[Theorem 3.3]{hall} that all $\C$-closed subgroups arise in this way.
\begin{lemma}[Classification of $\mathcal{C}$-closed subgroups]
	\label{closed_subgroups}
	Let $G$ be a group and let $H \leq G$. Then $H$ is $\mathcal{C}$-closed in $G$ if and only if it is an intersection of $\mathcal{C}$-open subgroups of $G$. 
\end{lemma}

\subsection{Restrictions of $\proC$ topologies}
Imagine the following situation: let $G$ be a group and let $H$ be its subgroup.  The $\proC$ topology induces a subspace topology on $H$, say $\mathcal{T}$. However, this topology might not necessarily be the same as $\proC$ topology on $H$: $\proC(H)$ will usually be finer than $\mathcal{T}$. For example, if $G$ is $F_2$, the free group on 2 generators and $H$ is $[G,G]$, the commutator subgroup of $G$, then $H$ contains only countably many subgroups that are open in $\mathcal{T}$; however, as $H$ is infinitely generated free group, $\NC(H)$ is uncountable.

For $H \leq G$ we say that $\proC(H)$ is \emph{induced} by $\proC(G)$ if $\proC(H)$ coincides with the subspace topology on $H$ induced by $\proC(G)$.

If $H\leq G$ and $X\subseteq H$ is $\mathcal{C}$-closed in $G$ then clearly $X$ is $\mathcal{C}$-closed in $H$. However, the implication in the other direction does not hold: let $G$ be the Baumslag-Solitar group BS(2,3) given by the presentation $\langle a,b \| ba^2b^{-1} = a^3 \rangle$. It is well known that this group is not residually finite. Take $H = \langle a \rangle \leq G$. Clearly $H \cong (\mathbb{Z},+)$, which is a residually finite group. Thus the singleton set $\{1\}$ is closed in the profinite topology on $H$ but it is not closed in the profinite topology on $G$ as $G$ is not residually finite. 

\begin{definition}
	Let $G$ be a group and let $H \leq G$ be its subgroup. We say that the pro-$\mathcal{C}(H)$ is a restriction of pro-$\mathcal{C}(G)$  if for all $X \subseteq H$ we have that $X$ is $\mathcal{C}$-closed in $H$ if and only if it is $\mathcal{C}$-closed in $G$.
\end{definition}

One can easily check that for $H \leq G$ $\proC(H)$ is a restriction of $\proC(G)$ if and only if $\proC(H)$ is induced by $\proC(G)$ and $H$ is $\C$-closed in $G$.

One type of subgroup on which the $\text{pro-}\mathcal{C}$ topologies behave well are retracts. Let $G$ be a group and let $R \leq G$. We say that $R$ is a \emph{retract} of $G$ if there is a homomorphism $\rho \colon G \to R$ such that $\rho \restriction_R = \id_R$. We will refer to $\rho$ as to the \emph{retraction} corresponding to $R$. We will often abuse the notation and consider $\rho$ as an endomorphism of $G$ as well. One could then equivalently say that $R$ is a retract of $G$ if and only if there is $\rho \colon G \to G$ such that $\rho \circ \rho = \rho$ and $R$ is the image of $\rho$.

Note that if $R \leq G$ is a retract then $G$ can be viewed as semidirect product $G = \ker(\rho) \rtimes R$, where $\rho \colon G \to R$ is the corresponding retraction. One can easily show the following by using the proof of \cite[Lemma 3.1.5]{profinite}.
\begin{lemma}
	\label{retract_restriction}
	Let $G$ be a residually-$\C$ group and let $R \leq G$ be a retract. Then $R$ is $\C$-closed in $G$ and $\proC(R)$ is induced by $\proC(G)$, hence $\proC(R)$ is a restriction of $\proC(G)$.
\end{lemma}

So far we have only been using assumptions that the class $\C$ satisfies (c1) and (c2). From now on we will also require the class $\C$ to satisfy (c4). The proof the next lemma follows from the proof of \cite[Lemma 3.1.4]{profinite}.
\begin{lemma}
	\label{open_restriction}
	Suppose that $\C$ is a class of finite groups satisfying (c1), (c2) and (c4). Let $G$ be a group and let $H \leq G$ be $\mathcal{C}$-open in $G$. Then pro-$\mathcal{C}(H)$ is a restriction of pro-$\mathcal{C}(G)$. 
\end{lemma}

As we mentioned earlier, the property of being residually-$\C$ is passed to subgroups. Obviously, the implication in the opposite direction does not hold: the group BS(2,3) contains a residually finite subgroup but BS(2,3) is not residually finite. However, the property of being residually-$\C$ is passed upwards from $\C$-open subgroups.
\begin{cor}
	\label{residually-C_open_subgroup}
	Let $\C$ be a class of finite groups satisfying (c1), (c2) and (c4). Let $G$ be a group and let $H \leq G$ be $\C$-open in $G$. If $H$ is residually-$\C$ then $G$ is residually-$\C$ as well.
\end{cor}
\begin{proof}
	The singleton set $\{1\}$ is $\C$-closed in $H$ as $H$ is residually-$\C$. By Lemma \ref{open_restriction} we see that the singleton set $\{1\}$ is $\C$-closed in $G$ as the $\proC(H)$ is a restriction of $\proC(G)$. Hence, we see that the group $G$ is residually-$\C$.
\end{proof}

The structure of classes of finite groups that satisfy (c1) and (c2) only can be quite wild. However, what if we also require the class $\C$ to satisfy (c4)? Suppose that $\C$ is a class of finite groups satisfying (c1), (c2) and (c4) and suppose that there is a nontrivial group $F$ such that $F \in \C$. Let $n = |F|$ and let $p$ be a prime number such that $p$ divides $n$. Clearly there is $g \in F$ such that $\text{ord}(g) = p$ and thus $F$ contains $C_p$, the cyclic group of size $p$ as a subgroup and thus $C_p \in \C$ as well. We see that the class $\C$ contains all finite $p$-groups as every finite $p$-group can be constructed from $C_p$ by a series of extensions. It is well known fact (see \cite{gruenberg}) that free groups are residually-$p$ for every prime number $p$ and therefore free groups are residually-$\C$ whenever the class $\C$ satisfies (c1), (c2) and (c4) and contains at least one nontrivial group.

\begin{lemma}
\label{free-by-C_is_residually-C}
	Let $\C$ be a class of finite groups satisfying (c1), (c2) and (c4) and assume that $\C$ contains a nontrivial group. Then free-by-$\C$ groups are residually-$C$.
\end{lemma}
\begin{proof}
	Let $G$ be a free-by-$\C$ group. By assumption there is $N \in \NC(G)$ such that $N$ is free. Clearly $N$ is $\C$-open in $G$. By previous argumentation we know that $N$ is residually-$C$ and hence $G$ is residually-$\C$ by Corollary \ref{residually-C_open_subgroup}.
\end{proof}



%% file: graph_products.tex
\section{Graph products of groups}
\label{graph_products}
Let $\Gamma$ be a simplicial graph and let $\mathcal{G} = \{G_v \mid v \in V\Gamma\}$ be a family of groups indexed by $V\Gamma$. Recal that the graph product $\Gamma \mathcal{G}$ is the quotient of the free product $\bigast_{v \in V\Gamma}G_v$ obtained by adding relations of the form
	\begin{displaymath}
		g_u g_v = g_v g_u \mbox{ for all $g_u \in G_u, g_v \in G_v$ such that $\{u,v\} \in E\Gamma$}.
	\end{displaymath}

For $v \in V \Gamma$ we define the $\link(v)$ to be the the set of vertices adjacent to $v$ in $\Gamma$ (excluding $v$ itself), more precisely 
$\link(v)=\{u \in V\Gamma \mid \{u,v\} \in E\Gamma\}$.
For a subset $A \subseteq V\Gamma$ we define $\link(A) = \bigcap_{v\in A} \link(v)$.

For $v \in V \Gamma$ we define the $\sstar(v)$ to be the the set of vertices adjacent to $v$ in $\Gamma$ including $v$, more precisely 
$\sstar(v)=\{u \in V\Gamma \mid \{u,v\} \in E\Gamma\} \cup \{v\}$.
For a subset $A \subseteq V\Gamma$ we define $\sstar(A) = \bigcap_{v\in A} \sstar(v)$.

Let $G = \Gamma\mathcal{G}$ be a graph product. Then every $g \in G$ can be obtained as a product of a sequence of generators $W \equiv (g_1, g_2, \dots , g_n)$ where each $g_i$ belongs to some $G_{v_i} \in \mathcal{G}$. We will say that $W$ is a \emph{word} in $G$ and the $g_i$ are its \emph{syllables}. The number of syllables is the \emph{length} of a word.

Transformations of the three following types can be defined on words in graph products:
\begin{enumerate}
	\item[(T1)]	remove a syllable $g_i$ if $g_i =_{G_v} 1$, where $v \in V\Gamma$ and $g_i \in G_v$,
	\item[(T2)]	remove two consecutive syllables $g_i, g_{i+1}$ belonging to the same vertex group $G_v$ and replace them by a single syllable $g_i g_{i+1} \in G_v$,
	\item[(T3)] interchange consecutive syllables $g_i \in G_u$ and $g_{i+1} \in G_v$ if $\{u, v\} \in E\Gamma$.
\end{enumerate}
The last transformation is also called \emph{syllable shuffling}. Note that the transformations (T1) and (T2) decrease the length of a word whereas (T3) preserves it. Thus by applying these transformations to a word $W$ we will eventually get a word $W'$ of minimal length representing the same element in $G$.

For $1 \leq i < j \leq n$ we say that syllables $g_i, g_j$ can be \emph{joined together} if they belong to the same vertex group and \lq everything in between commutes with them\rq. More formally: $g_i, g_j \in G_v$ for some $v \in V\Gamma$ and for all $i < k < j$ we have $g_k \in G_{v_k}$ such that $v_k \in \link(v)$. In this case obviously the words $W \equiv (g_1, \dots, g_{i-1}, g_i, g_{i+1}, \dots, g_{j-1}, g_j, g_{j+1}, \dots, g_n)$ and $W' \equiv (g_1, \dots, g_{i-1}, g_i g_j, g_{i+1}, \dots, g_{j-1}, g_{j+1}, \dots, g_{n})$ represent the same group element in $G$, but the word $W'$ is strictly shorter than $W$.

We say that a word $W \equiv (g_1, g_2, \dots , g_n)$ is \emph{$\Gamma$-reduced} if it is either an empty word or if $g_i \neq 1$ for all $i$ and no two distinct syllables can be joined together. To avoid any confusion with the terminology of special amalgams (see Section \ref{special_amalgams}) we will be using the Greek letter $\Gamma$ to emphasise that we are considering sequences and elements (cyclically) reduced in the context of graph products and not in the context of special amalgams.

As it turns out, the notion of being $\Gamma$-reduced and the notion of having minimal length coincide: the following theorem was proved by E. Green \cite[Theorem 3.9]{green}.
\begin{thrm}[Normal Form Theorem]
	\label{nft}
	Let $G = \Gamma\mathcal{G}$ be a graph product. Every element $g \in G$ can be represented by a $\Gamma$-reduced word. Moreover, if two $\Gamma$-reduced words $W, W'$ represent the same element in the group $G$ then $W$ can be obtained from $W'$ by a finite sequence of syllable shuffling. In particular, the length of a $\Gamma$-reduced word is minimal among all words representing $g$, and a $\Gamma$-reduced word represents the trivial element if and only if it is the empty word.
\end{thrm}

Thanks to Theorem \ref{nft} the following definitions make sense. Let $g \in G$ and let $W = (g_1,\dots, g_n)$ be a $\Gamma$-reduced word representing $g$. We define the \emph{length} of $g$ in $G$ to be $|g| = n$
and the \emph{support} of $g$ in $G$ to be
 $$\supp(g) = \{v \in V\Gamma | \exists i \in \{1,\dots, n\} \mbox{ such that } g_i \in G_v\setminus \{1\}\}.$$
  
 We define $\FL(g) \subseteq V\Gamma$ as the set of all $v \in V\Gamma$ such that there is a $\Gamma$-reduced word $W$ that represents the element $g$ and starts with a syllable from $G_v$. Similarly we define $\LL(g) \subseteq V\Gamma$ as the set of all $v \in V\Gamma$ such that there is a $\Gamma$-reduced word $W$ that represents the element $g$ and ends with a syllable from $G_v$. Note that $\FL(g) = \LL(g^{-1})$.
 
Let $x, y \in G$ and let $W_x \equiv (x_1, \dots ,x_n)$ be a $\Gamma$-reduced expression for $x$ and let $W_y \equiv (y_1 \dots, y_m)$ be a $\Gamma$-reduced expression for $y$. We say that the product $xy$ is a\emph{$\Gamma$-reduced product} if the word $(x_1, \dots, x_n, y_1, \dots, y_m)$ is $\Gamma$-reduced. Obviously, $xy$ is a $\Gamma$-reduced product if and only if $|xy| = |x| + |y|$. Or equivalently we could say that $xy$ is $\Gamma$-reduced product if and only if $\LL(x) \cap \FL(y) = \emptyset$. We can naturally extend this definition: for $g_1, \dots, g_n \in G$ we say that the product $g_1 \dots g_n$ is \emph{$\Gamma$-reduced} if $|g_1 \dots g_n| = |g_1| + \dots + |g_n|$.

\subsection{Full and parabolic subgroups}
\label{full_subgroups}
Let $\Gamma$ be a graph. For any subset $A \subseteq V\Gamma$ we will denote the corresponding full subgraph by $\Gamma_A$: $V \Gamma_A = A$ and for $u,v \in A$ we have $\{u,v\} \in E\Gamma_A$ if and only if $\{u,v\} \in E\Gamma$.

Let $A \subseteq V\Gamma$ and let $G_A$ denote the subgroup of $G$ generated by all $G_v$, where $v\in A$. Using Theorem \ref{nft} one can easily check that $G_A$ is isomorphic to the graph product $\Gamma_A \mathcal{G}_A$, where $\mathcal{G}_A = \{G_v \mid v \in A\}$. We see that every $A \subset V\Gamma$ induces a subgroup $G_A \leq G$. We will call subgroups of this type \emph{full subgroups} of $\Gamma\mathcal{G}$, and we say that a full subgroup $G_A$ is a \emph{proper} full subgroup if $A$ is a proper subset of $V\Gamma$. We say that a full subgroup $G_A \leq G$ is \emph{maximal} if $|V\Gamma \setminus A| = 1$. By definition $G_{\emptyset} = \{1\}$ is also a full subgroup corresponding to the empty subset of $V\Gamma$.

It can be easily seen that full subgroups are actually retracts.
\begin{remark}
	\label{full_is_retract}
	Let $G = \Gamma\mathcal{G}$ be a graph product of groups and let $G_A \leq G$ be a full subgroup. Then $G_A$ is a retract in $G$ with corresponding retraction map $\rho_A: G \rightarrow G_A$ defined on generators of G as follows:
	\begin{displaymath}
		\rho_A(g) = \left\{ 
			\begin{array}{ll}
				g & \mbox{ if } g \in G_v \mbox{ and } v\in A,\\
				1 & \mbox{otherwise.}
			\end{array} \right.
	\end{displaymath}
\end{remark}

Let $A,B \subseteq V\Gamma$ be arbitrary. Let $G_A, G_B \leq G$ be the corresponding full subgroups and let $\rho_A, \rho_B \in \End(G)$ be the corresponding retractions. We see that $\rho_A$ and $\rho_B$ commute: $\rho_A \circ \rho_B = \rho_B \circ \rho_A$. It follows that $G_A \cap G_B = G_{A\cap B}$ and $\rho_A \circ \rho_B = \rho_{A \cap B}$. This result can be generalised and strengthened.

Let $K \leq G$. We say that $K$ is \emph{parabolic} subgroup of $G$ if $K$ is conjugate to a full subgroup, i.e., if there are $A \subseteq V\Gamma$ and $g \in G$ such that $K = g G_A g^{-1}$. As it turns out, the intersection of parabolic subgroups is again a parabolic subgroup. The following theorem was proved in \cite[Corollary 3.6]{yago}. 
\begin{thrm}
	\label{intersection}
	Let $G = \Gamma\mathcal{G}$ be a graph product and let $K, L \leq G$ such that $K = g G_A g^{-1}$ and $H = f G_B f^{-1}$, where $A, B \subseteq V\Gamma$ and $f, g \in G$. Then there is $h \in G$ and $C \subseteq A \cap B$ such that $K \cap L = h G_C h^{-1}$.
\end{thrm}

As an easy consequence of Theorem \ref{intersection} we get the following lemma.
\begin{lemma}
	\label{maximal_full_subgroup}
Let $g \in G = \Gamma\mathcal{G}$ and suppose that $|V\Gamma| < \infty$ and $g \neq 1$. Then there is a maximal full subgroup $A$ of $G$ such that $g \not \in A^{G}$.
\end{lemma}
\begin{proof}
	Let $g \in G \setminus \{1\}$ and assume that the statement of the lemma does not hold for $g$, thus for every maximal full subgroup $A_v$, where $A_v = G_{V\Gamma \setminus \{v\}}$ for some $v \in V\Gamma$, there is $h_v \in G$ such that $g \in h_v A_v h_v^{-1}$. We see that $g \in \bigcap_{v \in V\Gamma}h_v A_v h_v^{-1}$. By Theorem \ref{intersection} we see that there are $h \in G$ and $C \subseteq \bigcap_{v \in V\Gamma}V\Gamma \setminus \{v\}$ such that $g \in h G_C h^{-1}$. However, $\bigcap_{v \in V\Gamma}V\Gamma \setminus \{v\} = \emptyset$ and thus we see that $g = 1$, which is a contradiction because we assumed that $g \neq 1$.
\end{proof}

The following theorem was proved in \cite[Proposition 3.10]{yago}.

\begin{thrm}
\label{parabolic_closure_well-defined}
	Let $X$ be a subset of the graph product $G = \Gamma \mathcal{G}$ such that at least one of the following conditions holds:
	\begin{enumerate}
		\item[(i)] the graph $\Gamma$ is finite;
		\item[(ii)] the subgroup $\langle X\rangle \leq G$ is finitely generated.
	\end{enumerate}
	Then there exists a unique minimal parabolic subgroup of $G$ containing $X$.
\end{thrm}

Suppose that a subset $X \subseteq G$ is contained in a minimal parabolic subgroup of $G$. Then this subgroup will be called the \emph{parabolic closure} of $X$ and will be denoted by $\PC_{\Gamma}(X)$.

For a subset $X \subseteq$ and a subgroup $H$ we will use $N_H(X)$ to denote $\{ g \in G \mid gX = Xg\}$, the $H$-normaliser of $X$ in $G$. The following characterisation of normalisers of parabolic subgroups was given in \cite[Proposition 3.13]{yago}.

\begin{thrm}
\label{normalisers_of_parabolic_subgroups}
Let $K$ be a nontrivial parabolic subgroup of the graph product $G = \Gamma \mathcal{G}$. Choose $f \in G$ and $S \in V\Gamma$ such that $K = f G_S f^{-1}$ and $G_s \neq \{1\}$ for all $s \in S$. Then $N_G(K) = f G_{S\cup \link(S)}f^{-1}$; in particular the normaliser $N_G(K)$ is a parabolic subgroup of $G$.
\end{thrm}

For a subset $X \subseteq$ and a subgroup $H$ we will use $C_H(X)$ to denote $\{ g \in G \mid gx = xg\}$, the $H$-centraliser of $X$ in $G$. Centralisers in graph products were fully described by Barkauskas in \cite{barkauskas}. We give simple a lemma describing centralisers of elements in terms of certain special subgroups and centralisers in full subgroups.
\begin{lemma}
\label{centralisers_in_graph_products}
	Let $G = \Gamma \mathcal{G}$ be a graph product of groups and let $g \in G$ be arbitrary. Suppose that there is $A \subseteq G$ such that $\PC_{\Gamma}(\langle g \rangle) = G_A$. Then $C_G(g) = C_{G_A}(g)G_{\link(A)}$.
\end{lemma}
\begin{proof}
	Clearly $C_G(g) \leq N_G(\langle g\rangle)$. Since $G_A = \PC_{\Gamma}(\langle g \rangle)$ we see by \cite[Lemma 3.12]{yago} that $N_G(\langle g \rangle) \subseteq N_G(G_A)$. By Theorem \ref{normalisers_of_parabolic_subgroups} we see that $N_G(G_A) = G_A \cdot G_{\link(A)}$ and thus $C_G(g) \subseteq G_A \cdot G_{\link(A)}$. We can then write $C_G(g) = C_G(g) \cap G_A \cdot G_{\link(A)}$. Note that $G_{\link(A)} \leq C_{G}(G_A)$ and thus $G_{\link(A)} \leq C_G(g)$. This means that $C_G(g) \cap G_A \cdot G_{\link(A)} = (C_G(g) \cap G_A)G_{\link(A)}$ and we see that $C_G(g) = C_{G_A}(g)G_{\link(A)}$.
\end{proof}

\subsection{Cyclically reduced elements and conjugacy in graph products}

Let $g \in G$, let $W \equiv (g_1, \dots ,g_n)$ be a $\Gamma$-reduced expression for $g$. We say that a sequence $W' = (g_{j+1}, \dots, g_n, g_{1}, \dots, g_j)$, where $j \in \{1,\dots, n-1\}$, is a \emph{cyclic permutation} of $W$. We say that the element $g'\in G$ is a \emph{cyclic permutation} of $g$ if $g'$ can be expressed by a cyclic permutation of some $\Gamma$-reduced expression for $g$.
 
Let $W \equiv (g_1, \dots, g_n)$ be some reduced expression in $G$. We say that $W$ is \emph{$\Gamma$-cyclically reduced} if all cyclic permutations of $W$ are $\Gamma$-reduced. We would like to extend this definition to elements of $G$. However, to achieve that we first need to show that this property does not depend on the choice of $\Gamma$-reduced expression.

\begin{lemma}
	Let $g \in G$ be arbitrary and let $W \equiv (g_1, \dots g_n)$ be some $\Gamma$-reduced expression for $g$. If $W$ is $\Gamma$-cyclically reduced then all $\Gamma$-reduced expressions representing $g$ are $\Gamma$-cyclically reduced.
\end{lemma}
\begin{proof}
	Assume that $W = (g_1, \dots, g_n)$ is $\Gamma$-cyclically reduced sequence and let $i \in \{1, \dots, n-1\}$ be arbitrary such that $[g_i, g_{i+1}] = 1$. Consider the expression $W'=(g_1, \dots, g_{i-1}, g_{i+1}, g_i, g_{i+2}, \dots , g_n)$. Obviously $W'$ is a $\Gamma$-reduced expression for $g$ as well. Let $W''$ be some cyclic permutation of $W'$. Then there are three cases to consider:
	\begin{enumerate}
		\item[(i)] $W'' = (g_{j+1}, \dots, g_{i-1}, g_{i+1}, g_i, g_{i+2},\dots , g_n, g_1, \dots, g_j)$ for some $j < i$,
		\item[(ii)] $W'' = (g_i, g_{i+2}, \dots, g_{n}, g_1, \dots, g_{i-1}, g_{i+1})$,
		\item[(iii)] $W'' = (g_{j+1}, \dots, g_n, g_1, \dots, g_{i-1}, g_{i+1}, g_i, g_{i+2}, \dots , g_j)$ for some $j > i$.
	\end{enumerate}
	
	Consider the case (i) first. The expression $W''$ can be rewritten to the expression $V = (g_{j+1}, \dots, g_n, g_1, \dots,g_j)$ by swapping the syllables $g_i$ and $g_{i+1}$. We see that $V$ is $\Gamma$-reduced as it is a cyclic permutation of $W$ and $W$ is $\Gamma$-cyclically reduced by assumption. It follows by Theorem \ref{nft} that $W''$ is $\Gamma$-reduced as it represents the same element as $V$ and both $W''$ and $V$ are of the same length. The case (iii) can be dealt with similarly.
	
	For case (ii) we see that the segment $(g_{i+2}, \dots, g_{n}, g_1, \dots, g_{i-1})$ is $\Gamma$-reduced as it is a segment of a cyclic permutation of $W$ and $W$ is $\Gamma$-cyclically reduced. Suppose that the sequence $W''$ is not $\Gamma$-reduced. Suppose that $g_i$ can be joined with $g_k$, where $k \in  \{i+2, \dots, n\}$. If this was the case then the syllable $g_{i}$ could have been joined with $g_k$ in $W$ which is a contradiction with our assumption that $W$ is $\Gamma$-reduced. Suppose that the syllable $g_i$ can be joined with $g_l$, where $l \in \{1, \dots, i-1\}$. Since the syllables $g_i$ and $g_{i+1}$ commute we see that $g_i$ and $g_l$ could be joined in the expression $P = (g_i, g_{i+1}, \dots, g_n, g_1, \dots, g_{i-1})$. However, $P$ is a cyclic permutation of $W$ and therefore $P$ is $\Gamma$-reduced as $W$ is $\Gamma$-cyclically reduced by assumption. By a similar argumentation we can show that the syllable $g_{i+1}$ cannot be joined with any of the syllables $g_1, \dots, g_{i-1}$ or $g_{i+1},\dots, g_n$. Clearly $g_i$ cannot be joined with $g_{i+1}$ as we assume that $W$ is $\Gamma$-reduced. Therefore we see that $W''$ is $\Gamma$-reduced.
	
	We have shown that the property of being $\Gamma$-cyclically reduced is preserved by transformation (T3). By Theorem \ref{nft} every $\Gamma$-reduced expression for $g$ can be obtained from $W$ by a finite sequence of transformation of type (T3). Hence all $\Gamma$-reduced expressions for $g$ are $\Gamma$-cyclically reduced.
\end{proof}

As a direct consequence of this lemma we see that a $\Gamma$-reduced expression $W = (g_1, \dots, g_n)$ is $\Gamma$-cyclically reduced if and only if the following condition is satisfied: let $i,j \in \{1\dots, n\}$ be such that $g_i$ can be shuffled to the beginning of $W$ and $g_j$ can be shuffled to the end of $W$ and $g_i$ and $g_j$ belong to the same vertex group; then $i=j$.

\begin{definition}
	Let $g \in G$ be arbitrary. We say that $g$ is $\Gamma$-cyclically reduced if either $g$ is trivial or some $\Gamma$-reduced word representing $g$ is $\Gamma$-cyclically reduced. 
\end{definition}

Note that $\FL(g)\cap\LL(g) \neq \emptyset$ does not necessarily mean that $g$ is not $\Gamma$-cyclically reduced. Suppose that $\supp(g) \cap \sstar(\supp(g)) \neq \emptyset$. Then there is $v \in \supp(g)$ such that it is connected with all the other vertices in $\supp(g)$. This means that there is $i \in  \{1, \dots, n\}$ such that the syllable $g_i$ commutes with all the other syllables and can be shuffled to both ends of $g$, thus $v \in \FL(g) \cap \LL(g)$.

\begin{definition}[P-S decomposition]
	Let $g \in G$. We define $S(g) = \supp(g) \cap \sstar(\supp(g))$. Similarly, we define $P(g) = \supp(g) \setminus S(g)$. Obviously $g$ uniquely factorises as a $\Gamma$-reduced product $g = s(g) p(g)$ where $\supp(s(g)) = S(g)$ and $\supp(p(g)) = P(g)$. We call this factorisation the P-S decomposition of $g$.
\end{definition}

Note that $\FL(g) = S(g) \dot\cup \FL(p(g))$, $\LL(g) = S(g) \dot\cup \LL(p(g))$ and $S(p(g)) = \emptyset$. Another simple observation is that if $g'$ is a cyclic permutation of $g$ then $g'$ can be uniquely factorised as $s(g) p'$, where $p'$ is a cyclic permutation of $p(g)$.

\begin{lemma}
	\label{graph_product_cyclically_reduced_criterion}
	Let $g \in G$. Then the following are equivalent:
	\begin{enumerate}
		\item[(i)]	$g$ is $\Gamma$-cyclically reduced,
		\item[(ii)]	$(\FL(g)\cap \LL(g)) \setminus S(g) = \emptyset$,
		\item[(iii)]$\FL(p(g)) \cap \LL(p(g)) = \emptyset$,
		\item[(iv)] $p(g)$ is $\Gamma$-cyclically reduced.
	\end{enumerate}
\end{lemma}
\begin{proof}
	$\text{(i)} \Rightarrow \text{(ii)}$: assume that $g$ is $\Gamma$-cyclically reduced. Let $(g_1, \dots, g_n)$ be some $\Gamma$-reduced expression for $g$. Without loss of generality we may assume that $s(g) = g_1 \dots g_s$ and $p(g) = g_{s+1} \dots g_n$, where $s = |S(g)|$. Suppose that $v \in (\FL(g)\cap \LL(g)) \setminus S(g)$. Then there are $1 \leq i < j \leq n$ such that $g_i, g_j \in G_v$. Since $v \in \FL(g)$ we see that $g_i$ can be shuffled to beginning of $g$. Similarly $g_j$ can be shuffled to the end of $g$ and hence 
	\begin{displaymath}
		W = (g_i, g_1, \dots ,g_{i-1}, g_{i+1}, \dots, g_{j-1}, g_{j+1}, \dots, g_n, g_j)
	\end{displaymath}
	is also a $\Gamma$-reduced expression for $g$. However, the expression 
	\begin{displaymath}
		W' = (g_j, g_i, g_1, \dots, g_{i-1}, g_{i+1}, \dots, g_{j-1}, g_{j+1}, \dots, g_n)
	\end{displaymath}
	is not reduced which is a contradiction as $W'$ is a cyclic permutation of $W$ and $g$ is $\Gamma$-cyclically reduced.
	
	$\text{(ii)}\Rightarrow \text{(iii)}$: suppose that $(\FL(g)\cap \LL(g)) \setminus S(g) = \emptyset$. As mentioned before, $\FL(g) = \FL(p(g)) \dot\cup S(g)$ and $\LL(g) = \LL(p(g)) \dot\cup S(g)$ and therefore $\FL(p(g)) \cap \LL(p(g)) = \emptyset$.
	
	$\text{(iii)} \Rightarrow \text{(iv)}$: if $\FL(p(g)) \cap \LL(p(g)) = \emptyset$ then clearly $p(g)$ is $\Gamma$-cyclically reduced.
	
	$\text{(iv)} \Rightarrow \text{(i)}$: assume that $p(g)$ is $\Gamma$-cyclically reduced and there is $ v \in \FL(p(g)) \cap \LL(p(g))$. Let $(p_1, \dots, p_m)$ be a $\Gamma$-reduced expression for $p(g)$. Suppose that there are $1 \leq i<j \leq m$ such that $p_i, p_j \in G_v$. This is clearly a contradiction since $p(g)$ is $\Gamma$-cyclically reduced by assumption. This means that there is $i \in \{1, \dots, n\}$ such that the expression $(p_1, \dots, p_n)$ can be rewritten by shuffling to $(p_i, p_1, \dots, p_{i-1}, p_{i+1}, \dots, p_n)$ and to $(p_1, \dots, p_{i-1}, p_{i+1}, \dots, p_n, p_i)$ as well. This means that $p_i$ commutes with all the other syllables from $p(g)$ and hence the vertex $v$ is adjacent to all the vertices in $P(g)\setminus \{v\}$. But since $v$ is also connected to all the vertices in $S(g)$ by the definition of $S(g)$ we see that $v \in S(g)$. This is a contradiction as $v \in \supp(p(g))\subseteq P(g)$, hence we may assume that $\FL(p(g)) \cap \LL(p(g)) = \emptyset$. As stated before, $\FL(g) = \FL(p(g)) \dot{\cup} S(g)$ and $\LL(g) = \LL(p(g)) \dot{\cup} S(g)$. Since $\FL(p(g)) \cap \LL(p(g)) = \emptyset$ we see that $\FL(g)\cap \LL(g) = S(g)$. Let $W = (g_1, \dots, g_m)$ be a $\Gamma$-reduced expression for $g$. Suppose that there are $i,j \in \{1,\dots,n\}$ such that $g_i$ can be shuffled to the beginning of $W$, $g_j$ can be shuffled to the end of $g$ and $g_i$ and $g_i$ belong to the same vertex group. Since $\FL(g)\cap \LL(g) = S(g)$ we see that $g_i, g_j \in G_s$ for some $s\in S(g)$ as $W$ is $\Gamma$-reduced. This means that $i=j$ and consequently that $g$ is $\Gamma$-cyclically reduced.
\end{proof}

\begin{lemma}[Conjugacy criterion for graph products]
	\label{conjugacy_criterion_for_graph_products}
	Let $x, y$ be $\Gamma$-cyclically reduced elements of $G = \Gamma \mathcal{G}$. Then $x \sim_G y$ if and only if the all of the following are true:
	\begin{enumerate}
		\item[(i)] $|x| = |y|$ and $\supp(x) = \supp(y)$,
		\item[(ii)] $p(x)$ is a cyclic permutation of $p(y)$,
		\item[(iii)] $s(y) \in s(x)^{G_{S(x)}}$.
	\end{enumerate}
\end{lemma}
\begin{proof}
	The "if" part of the claim holds trivially.
	
	Let $x, y \in G$ be $\Gamma$-cyclically reduced such that $x \sim_G y$. Without loss of generality we will assume that $|x| \geq |y|$. Let $X \subseteq G$ denote the set of all cyclic permutations of $x$. Clearly $X^{G_{S(x)}} \subseteq x^G$. Pick $x' \in X^{G_{S(x)}}$ such that the corresponding $g' \in G$, where $g'x'{g'}^{-1} = y$, is of minimal length. First, we show by induction on $|\LL(g') \cap \FL(x')|$ that there  are elements $x'' \in X^{G_{S(x)}}$ and $g'' \in G$ such that $|g''| = |g'|$, $g''x''{g''}^{-1} = y$ and the product $g''x''$ is $\Gamma$-reduced. If $|\LL(g') \cap \FL(x')| = 0$ then clearly the product $g'x'$ is $\Gamma$-reduced and the claim holds for $g'' = g'$ and $x'' = x'$. Suppose that $|\LL(g) \cap \FL(x')| = c > 0$ and that the statement holds for all $c' < c$. Let $(g_1,\dots,g_k)$ be a $\Gamma$-reduced expression for $g'$ and let $(x_1, \dots, x_n)$ be a $\Gamma$-reduced expression for $x'$. Without loss of generality we may assume that $g_k$ and $x_1$ belong to the same vertex group, say $G_v$. Then
	\begin{align*}
		y 	&= g_1 \dots g_k x_1 \dots x_n g_k^{-1} \dots g_1^{-1}\\
			&= g_1 \dots g_{k-1}(g_k x_1) x_2 \dots x_n x_1 (g_k x_1)^{-1} g_{k-1}^{-1} \dots g_1^{-1}.
	\end{align*}
 Obviously $g_k \neq x_1^{-1}$ as otherwise we could replace $x'$ by $x_2 \dots x_n x_1$, a cyclic permutation of $x$, and $g'$ by $g_1 \dots g_{k-1}$. Clearly $x_2 \dots x_n x_1 \in X^{G_{S(X)}}$ and
 \begin{displaymath}
 	g_1 \dots g_{k-1} x_2 \dots x_n x_1 g_{k-1}^{-1} \dots g_1^{-1} = y
 \end{displaymath}
 which is a contradiction with our choice of $x'$ and $g'$ as $|g_1 \dots g_{k-1}| < |g'|$. If $v \in S(x)$ then $g_k x g_k^{-1} \in X^{G_{S(x)}}$ and again we have a contradiction with our choice of $x'$ and $g'$. We see that $v \not\in S(x)$ and thus $\LL(g') = \LL(g' x_1)$ and also $v \not\in \FL(x_2\dots x_n x_1)$. Note that if $g_i$ can be shuffled to the end of $g'$ then $[g_i,g_k] = 1$ and necessarily $\{u,v\} \in E\Gamma$, where $g_i \in G_u$. If $w \in \FL(x_2 \dots x_n x_1) \setminus \FL(x')$ then we see that $\{v,w\}\not\in E\Gamma$ hence $w \not\in \LL(g'x_1)$. From this we can conclude that $v \not\in \LL(g'x_1) \cap \FL(x_2 \dots x_n) \subseteq \LL(g')\cap \FL(x')$ hence $\LL(g' x_1) \cap \FL(x_2 \dots x_n x_1)$ is a proper subset of $\LL(g') \cap \FL(x')$. Now we can use induction hypothesis and we are done.
 
	We have $g'' x'' = y g''$. Since $g'' x''$ is a $\Gamma$-reduced product we see that $|g'' x''| = |g''| + |x''| = |g| + |x| = n + k$. Also $|y g''| \leq |y| + |g| = m + k$, where $m = |y|$. However, we assumed that $|x| \geq |y|$ and thus we see that $n = m$ and consequently $y g''$ is a $\Gamma$-reduced product as well. Let $(y_1, \dots, y_n)$ be some $\Gamma$-reduced expression for $y$ and suppose that $(x_1, \dots, x_n)$ and $(g_1, \dots, g_k)$ are $\Gamma$-reduced expressions for $x''$ and $g''$. We have 
	\begin{displaymath}
		g_1 \dots g_k x_1 \dots x_n g_k^{-1} \dots g_1^{-1} = y_1 \dots y_n.
	\end{displaymath}
	The expression $(g_1, \dots, g_k, x_1,\dots, x_n, g_k^{-1}, \dots, g_1^{-1})$ cannot be $\Gamma$-reduced by Theorem \ref{nft}. Assume that the syllable $g_k^{-1}$ can be joined up with $g_j$ for some $j \in \{1, \dots, k\}$. But then by definition $[g_k^{-1}, x_i]=1$ for all $i=1,\dots, n$. Clearly 
	\begin{displaymath}
		g_1 \dots g_{k-1} x_1 \dots x_n g_{k-1}^{-1} \dots g_1^{-1} = y_1 \dots y_n.
	\end{displaymath}
	which is a contradiction with the minimality of $|g|$.
	 Since $g'' x''$ is a $\Gamma$-reduced product we then see that the expression $(x_1, \dots, x_n, g_k^{-1}, \dots g_1^{-1})$ is not $\Gamma$-reduced. Without loss of generality we may assume that $g_k$ and $x_n$ belong to the same vertex group. Assume that $g_k \neq x_n$. Then we have
	\begin{displaymath}
		g_1 \dots g_k x_1 \dots x_{n-1}(x_n g_k^{-1}) = y_1 \dots y_n g_1 \dots g_{k-1}.
	\end{displaymath}
	From the construction of $x''$ and $g''$ we see that $(g_1, \dots ,g_n, x_1, \dots ,x_{n-1},x_n g_k^{-1})$ is a $\Gamma$-reduced expression and so is $(y_1, \dots ,y_n ,g_1, \dots ,g_{k-1})$. However, this is a contradiction with Theorem \ref{nft} as both of these expressions represent the same group element, but they are not of the same length. Hence we see that $g_k = x_n$
	\begin{displaymath}
		y_1 \dots y_n = g_1 \dots g_{k-1} x_n x_1 \dots x_{n-1} g_{k-1}^{-1} \dots g_1^{-1}
	\end{displaymath} 
	which is a contradiction as we could replace $x_n x_1 \dots x_{n-1}$, a cyclic permutation of $x''$ and thus element of $X^{G_{S(x)}}$ and $g''$ by $g_1 \dots g_{k-1}$ and get a shorter conjugator.
	We see that unless $g = 1$ we always get a contradiction. It follows that $y = x'' \in X^{G_{S(x)}}$ and consequently $\supp(x) = \supp(y)$, $s(y) \in s(x)^{G_{S(x)}}$ and $p(x)$ is a cyclic permutation of $p(y)$.
\end{proof}

%% file: centraliser_conditions.tex
\section{$\mathcal{C}$-centraliser conditions and $\mathcal{C}$-conjugacy separability}
\label{centraliser_conditions}
In this section we will assume that the class $\C$ satisfies (c1), (c2) and (c4), i.e. we will assume that the class $\C$ is closed under taking subgroups, direct products and extensions.

\begin{definition}
	We say that a group $G$ satisfies the $\mathcal{C}$-centraliser condition ($\mathcal{C}\text{-CC}$) if for every $K \in \mathcal{N_C}(G)$ and every $g \in G$ there is $L \in \mathcal{N_C}(G)$ such that $L \leq K$ and
	\begin{displaymath}
		C_{G/L}(\psi(g)) \subseteq \psi(C_G(g)K) \mbox{ in $G/L$},
	\end{displaymath}
	where $\psi \colon G \to G/L$ is the natural projection.
\end{definition}

Centraliser condition was introduced by Chagas and Zalesskii in \cite{chagas} in case when $\C$ is the class of all finite groups. However, their definition of centraliser condition was given in terms of profinite completion. They showed that if group $G$ is conjugacy separable and satisfies centraliser condition then $G$ is HCS. Minasyan gave the definition in terms of subgroups of finite index and showed that for residually finite groups the definitions are equivalent. Minasyan also showed that the implication in the other direction holds as well: CS group $G$ is HCS if and only if it satisfies CC (see \cite[Proposition 3.2]{raags}). Toinet proved that the same statement holds when $\C$ is the class of all finite $p$-groups for some $p\in \mathbb{P}$ (see \cite[Proposition 3.6]{toinet}). We show that the statement is true whenever the class $\C$ satisfies (c1), (c2) and (c4).

\begin{thrm}
	\label{cc_and_cs_is_hcs}
	Let $G$ be a group. Then the following are equivalent:
	\begin{itemize}
		\item[(i)] $G$ is $\mathcal{C}$-HCS,
		\item[(ii)] $G$ is $\mathcal{C}\text{-CS}$ and satisfies $\mathcal{C}\text{-CC}$. 
	\end{itemize}
\end{thrm}

Before we proceed with the proof of Theorem \ref{cc_and_cs_is_hcs} we need to define two more conditions.

\begin{definition}
	Let $G$ be a group and let $H \leq G$ and $g \in G$. We say that the pair $(H,g)$ satisfies the $\mathcal{C}$-centraliser condition in $G$ ($\mathcal{C}\text{-CC}_G$) if for every $K \in \mathcal{N_C}(G)$ there is $L \in \mathcal{N_C}(G)$ such that $L \leq K$ and
	\begin{displaymath}
		C_{\psi(H)}(\psi(g)) \subseteq \psi(C_H(g)K) \mbox{ in $G/L$},
	\end{displaymath}
	where $\psi \colon G \to G/L$ is the natural projection.
\end{definition}

Note that a group $G$ satisfies $\C$-CC if and only if the pair $(G,g)$ has $\C\text{-CC}_G$ for every $g \in G$.

\begin{definition}
	Let $G$ be a group and let $H \leq G$ be a subgroup. We say that that $H$ satisfies the $\mathcal{C}$-centraliser condition in $G$ ($\mathcal{C}\text{-CC}_G$) if the pair $(H,g)$ satisfies $\mathcal{C}\text{-CC}$ for every $g \in G$.
\end{definition}

Very often our proofs require case by case analysis. To keep the our proofs simple we will use the following lemma, which is a centraliser condition analogue of Lemma \ref{closed_simplification}.
\begin{lemma}
\label{cc_simplification}
		Let $G$ be a group and let $H \leq G$ and $g \in G$. Then the pair $(H,g)$ satisfies $\mathcal{C}\text{-CC}_G$ if and only if for every $K \in \mathcal{N_C}(G)$ there is a group $F$ and a surjective homomorphism $\phi: G \rightarrow F$, such that $\ker(\phi) \leq K$, the pair $(\phi(H),\phi(g))$ satisfies $\mathcal{C}\text{-CC}_F$ and
		\begin{displaymath}
			C_{\phi(H)}(\phi(g)) \subseteq \phi(C_H(g)K) \text{ in }F.
		\end{displaymath}
\end{lemma}
\begin{proof}
	Assume that the pair $(H,g)$ has $\mathcal{C}\text{-CC}_G$, thus for every $K \in \mathcal{N_C}(G)$ there is $L \in \mathcal{N_C}(G)$ such that $L \leq K$ and
	\begin{displaymath}
		C_{\psi(H)}(\psi(g)) \subseteq \psi(C_H(g)K) \mbox{ in $G/L$},
	\end{displaymath}
	where $\psi \colon G \to G/L$ is the natural projection. Then we can take $\phi = \text{id}_G$ and the statement clearly holds.
	
	To prove sufficiency let $K \in \mathcal{N_C}(G)$ be arbitrary. By assumption there is a group $F$ and a homomorphism $\phi\colon G \to F$ such that $\ker(\phi) \leq K$, $\phi(K) \in \mathcal{N_C}(F)$ and the pair $(\phi(H), \phi(g))$ satisfies $\mathcal{C}\text{-CC}_F$, thus there is $L' \in \mathcal{N_C}(F)$ such that $L' \leq \phi(K)$ and
	\begin{equation}
		\label{eqn}
		C_{\zeta(\phi(H))}(\zeta(\phi(g))) \subseteq \zeta(C_{\phi(H)}(\phi(g))\phi(K)) \mbox{ in $F/L'$},
	\end{equation}
	where $\zeta\colon F \to F/L'$ is the natural projection. Define $\psi\colon G \to F/L'$ to be given by $\psi = \zeta \circ \phi$. Set $L = \phi^{-1}(L')$. As $L' \leq \phi(K)$ and $\ker(\phi) \leq K$ we get that $L = \phi^{-1}(L') \leq K$. We see that $\phi^{-1}(L')=\ker(\psi) \in \mathcal{N_C}(G)$. Since $C_{\phi(H)}(\phi(g)) \subseteq \phi(C_H(g)K)$ in $F$ we see that
	\begin{displaymath}
		\zeta(C_{\phi(H)}(\phi(g))) \subseteq \zeta(\phi(C_H(g)K)) = \psi(C_H(g)K)
	\end{displaymath}
	and thus the equation (\ref{eqn}) can be rewritten to
	\begin{displaymath}
		C_{\psi(H)}(\psi(g)) \subseteq \psi(C_H(g)K) \mbox{ in $G/L$}.
	\end{displaymath}
	Since $K$ was arbitrary we see that the pair $(H,g)$ satisfies $\mathcal{C}\text{-CC}_G$.
\end{proof}

In order to be able to prove Theorem \ref{cc_and_cs_is_hcs} we will need the following three statements. All the proofs in this chapter (except for Lemma \ref{cc_simplification}) closely follow those given in \cite[Section 3]{raags}.

\begin{lemma}
	Let $G$ be a group, let $H \leq G$ and $g \in G$. Assume that the pair $(G,g)$ satisfies $\mathcal{C}\text{-CC}_G$ and the conjugacy class $g^G$ is $\mathcal{C}$-closed in $G$. If the double coset $C_G(g)H$ is $\mathcal{C}$-closed in $G$ then the set $g^H$ is also $\mathcal{C}$-closed in $G$.
\end{lemma}
\begin{proof}
	Let $y \in G \setminus g^H$ be arbitrary.
	
	If $y \not\in g^G$ then there is $L\in\mathcal{N_C}(G)$ such that $\phi(y) \not\in \phi(g^G)$ in $G/L$, where $\phi \colon G \to G/L$ is the natural projection, therefore $\phi(y) \not\in \phi(g^H)$.
	
	Assume that $y \in g^G \setminus g^H$, thus $y = zgz^{-1}$ for some $z \in G \setminus H$. Suppose $C_G(g) \cap z^{-1}H$ is nonempty, thus there is $f \in C_G(g)$ such that $zf \in H$. Then $g = fgf^{-1}$ and thus $y = zgz^{-1} = (zf)g(zf)^{-1} \in g^H$ which is a contradiction as we assume that $y \not\in g^H$, thus $C_G(g) \cap z^{-1}H = \emptyset$, in other words $z^{-1} \not\in C_G(g)H$. Since $C_G(g)H$ is $\mathcal{C}$-closed in $G$ by assumption, there is $K \in \mathcal{N_C}(G)$ such that $\{z^{-1}\} \cap C_G(g)HK = \emptyset$. Since the pair $(G,g)$ has $\mathcal{C}\text{-CC}_G$ by assumption, there is $L \in \mathcal{N_C}(G)$ such that $L \leq K$ and 
	\begin{displaymath}
		C_{G/L}(\phi(g)) \subseteq \phi(C_G(g)K),
	\end{displaymath}
	where $\phi \colon G \to G/L$ is the natural projection.
	
	Suppose that $\phi(y) \in \phi(g^H)$, thus there is some $h \in H$ such that $\phi(y) = \phi(zgz^{-1}) = \phi (hgh^{-1})$. We see that $\phi(z^{-1}h) \in C_{G/L}(\phi(g))$, thus 
	\begin{displaymath}
		\phi(z^{-1}) \in C_{G/L}(\phi(g))\phi(H) \subseteq \phi(C_G(g)K)\phi(H) = \phi(C_G(g)HK).
	\end{displaymath} This means that $z^{-1} \in C_G(g)HKL = C_G(g)HK$ as $L \leq K$. But that is a contradiction with the construction of $K$.
	
	He have showed that for arbitrary $y \in G \setminus g^H$ there is $L \in \mathcal{N_C}(G)$ such that $\phi(y) \not\in \phi(g^H)$, hence the set $g^H$ is $\mathcal{C}$-closed in $G$.
\end{proof}

\begin{cor}
	\label{cor3.5}
	Let $G$ be a $\mathcal{C}\text{-CS}$ group satisfying $\mathcal{C}\text{-CC}$ and let $H \leq G$ such that $C_G(h)H$ is $\mathcal{C}$-closed in $G$ for every $h \in H$. Then $H$ is $\mathcal{C}\text{-CS}$. Moreover, for every $h \in H$ the set $h^H$ is $\mathcal{C}$-closed in $G$.
\end{cor}

\begin{lemma}
	\label{3.7}
	Let $G$ be a group and suppose that $H \leq G$, $g \in G$ and $K \in \mathcal{N_C}(G)$. If the set $g^{H\cap K}$ is $\mathcal{C}$-closed in $G$ then there is $L \in \mathcal{N_C}(G)$ such that $L \leq K$ and
	\begin{displaymath}
		C_{\phi(H)}(\phi(g)) \subseteq \phi(C_H(G)K) \mbox{ in }G/L,
	\end{displaymath}
	where $\phi \colon G \to G/L$ is the natural projection.	
\end{lemma}
\begin{proof}
	Denote $k = |H : (H\cap K)| < \infty$. Then $H = \bigsqcup_{i=1}^{k}z_i(H \cap K)$ for some $z_1, \dots, z_k \in H$. If necessary we can renumber the elements $z_i$ so that there is $l \in \{0,1,\dots, k-1\}$ such that $z_i ^{-1}g z_i \not\in g^{H\cap K}$ if $1 \leq i \leq l$ and $z_i ^{-1}g z_i \in g^{H\cap K}$ if $l < i \leq k$. Since $g^{H\cap K}$ is $\mathcal{C}$-closed in $G$ there is $L \in \mathcal{N_C}(G)$ such that $\phi(z_i^{-1} g z_i) \not\in \phi(g^{H\cap K})$ in $G/L$, where $\phi \colon G \to G/L$ is the natural projection, for all $i = 1, 2, \dots, l$. Note that by replacing $L$ with $L \cap K$ we may assume that $L \leq K$.
	
	Let $\overline{x}\in C_{\phi(H)}(\phi(g))$ be arbitrary. Clearly $\overline{x} = \phi(x)$ for some $x \in H$ and thus $\phi(x^{-1}gx) = \phi(g)$ in $G/L$, therefore $x^{-1}gx \in gL$ in $G$. As $x \in H$ there is $i \in \{1, 2, \dots, k\}$ and $y \in H \cap K$ such that $x = z_i y$, thus $x^{-1}gx = y^{-1} z_i^{-1} g z_i y$. As a consequence we see that $z_i^{-1}gz_i \in y gLy^{-1} = ygy^{-1}L \subseteq g^{H\cap K}L$. This means that $\phi(z_i^{-1} g z_i) \in \phi(g^{H\cap K})$ in $G/L$ and thus from construction of $L$ we see that $l < i \leq k$, therefore $z_i^{-1} g z_i \in g^{H\cap K}$ and there is some $u \in H \cap K$ such that $z_i^{-1} x z_i = u g u^{-1}$. We see that $z_i u \in C_H(g)$ and since $x = z_i y = z_i u u^{-1} y$ we see that $x \in C_H(g)(H\cap K) \subseteq C_H(g)K$. This means that $\overline{x} \in \phi(C_H(g)K)$ in $G/L$. Since $\overline{x} \in C_{\phi(H)(\phi(g))}$ was arbitrary we see that
	\begin{displaymath}
		C_{\phi(H)}(\phi(g)) \subseteq \phi(C_H(g)K) \mbox{ in } G/L,
	\end{displaymath}
	which concludes the proof. 
\end{proof}

Now we are ready to prove the main statement of this chapter.
\begin{proof}[Proof of Theorem \ref{cc_and_cs_is_hcs}]
	$\text{(i)} \Leftarrow \text{(ii)}$: let $H \leq G$ be a $\mathcal{C}$-open subgroup of $G$. By Lemma $\ref{open_subgroups}$ we see that $H$ is of finite index in $G$, hence the double coset $C_G(h)H$ is a finite union of $\mathcal{C}$-closed sets and thus is $\mathcal{C}$-closed in $G$. By Corollary \ref{cor3.5} we see that $H$ is $\mathcal{C}\text{-CS}$.
	
	$\text{(i)} \Rightarrow \text{(ii)}$: assume that $G$ is $\mathcal{C}\text{-HCS}$. Let $g\in G$ and $K \in \mathcal{N_C}(G)$ be arbitrary. Let $H = \langle g \rangle K$ and note that $g^K = g^{H} = g^{H\cap K}$. Clearly $H$ is $\mathcal{C}$-open in $G$ by Lemma \ref{open_subgroups} and thus it is $\mathcal{C}\text{-CS}$. Since $g \in H$ and $H$ is $\mathcal{C}\text{-CS}$ we see that $g^H$ is $\mathcal{C}$-closed in $H$. By Lemma \ref{open_restriction} we see that $g^H = g^{H \cap K}$ is $\mathcal{C}$-closed in $G$. By previous lemma there is $L \in \mathcal{N_C}(G)$ such that $L \leq K$ and
	\begin{displaymath}
		C_{\phi(H)}(\phi(g)) \subseteq \phi(C_H(G)K) \mbox{ in }G/L,
	\end{displaymath}
	where $\phi \colon G \to G/L$ is the natural projection. Since $g \in G$ and $K \in \mathcal{N_C}(G)$ were arbitrary we see that $G$ has $\mathcal{C}\text{-CC}$.
\end{proof}
Note that we used the fact the class $\C$ is closed under extensions only in the proof of Theorem \ref{cc_and_cs_is_hcs} when we used Lemma \ref{open_restriction}. All the other statements in this chapter require only (c1) and (c2).

%% file: special_amalgams.tex
\section{Special amalgams}
\label{special_amalgams}

In order to be able to understand certain properties of graph products we will turn our attention to \emph{special amalgams}. The following section is a close analogue of \cite[Section 7]{raags}.
\begin{definition}
\label{special_amalgam}
	Let $A, C$ be groups and let $H \leq A$. Then we define $A \specam_H C$, the \emph{special amalgam} of $A$ and $C$ over $H$, to be the following free product with amalgamation:
	\begin{displaymath}
		A \ast_H (H \times C) \mbox{ given by presentation } \langle A,C \| [h,c] = 1 \ \forall h\in H ,\forall c \in C\rangle,
	\end{displaymath}
	where $[h,c] = hch^{-1}c^{-1}$.
\end{definition}

The main reason why we are interested in special amalgams is that they naturally appear in graph products.
\begin{remark}
	\label{structure_of_graph_products}
	Let $G =\Gamma\mathcal{G}$ be a graph product and suppose that $|V\Gamma|\geq 2$. Then for every $v \in V\Gamma$ there is a natural splitting of $G = G_A \specam_{G_C} G_B$ as a special amalgam of full subgroups, where $A = V \setminus \{v\}$, $B = \text{star}(v)$ and $C = \link(v)$. 
\end{remark}
\begin{proof}
	Let $v \in V\Gamma$ be arbitrary, set $A, C \subseteq V\Gamma$ as in the statement of the remark and let $B = \text{star}(v)$. Obviously $G_C \leq G_B$, $G_C \leq G_A$ and $G = \langle G_A, G_B\rangle$. By looking at the presentations it is easy to see that $G \cong G_A \ast_{G_C} G_B$. Note that the vertex $v$ is central in the graph $\Gamma_B$ therefore $G_B = G_v \times G_C$. Consequently $G \cong G_A\ast_{G_C}(G_C \times G_v) = G_A \specam_{G_C} G_v$.
\end{proof}

There are two extreme cases that can occur. If $v \in V\Gamma$ is an isolated vertex, i.e. $v$ is not connected to any other vertex, we see that $G_C = \{1\}$ and $G = G_A \ast G_v$. On the other side, if $\link(v)=V\Gamma \setminus \{v\}$, i.e. if $v$ is central in $\Gamma$, we see that $G = G_B = G_A \times G_v$.

\subsection{Normal form and functorial property}
Let $G = A \specam_{H}C$. Obviously every element $g \in G$ can be represented as a product $x_0 c_1 x_1 \dots c_n x_n$ where $x_i \in A$ for $i=0,1, \dots, n$ and $c_j \in C$ for $j = 1, \dots, n$. We say that $g = x_0 c_1 x_1 \dots c_n x_n$ is in a \emph{reduced form} if $x_i \not\in H$ for $i = 1, \dots, n-1$ and $c_j \neq 1$ for $j = 1, \dots, n$. By using the normal form theorem for free products with amalgamation \cite[Theorem 4.4]{magnus} we can prove the following.

\begin{lemma}
	\label{nft_special_amalgam}
	Let $H \leq A, C$ be groups and let $G = A \specam_H C$. Suppose that $g = x_0 c_1 x_1 \dots c_n x_n$, where $x_0, x_1, \dots, x_n \in A$ and $c_1, \dots, c_n \in C$, with $n \geq 1$ is in reduced form. Then $g \neq 1$ in G.
	
	Moreover, suppose that $f = y_0 d_1 y_1 \dots d_m y_m$, where $y_0, y_1, \dots, y_m \in A$ and $d_1, \dots, d_m \in C$, is in reduced form with $m \geq 1$ as well and $f = g$. Then $m = n$ and $c_i = d_i$ for all $i = 1, \dots, n$.
\end{lemma}
\begin{proof}
	The first assertion of the lemma follows directly from normal form theorem for free products with amalgamation.
	Now, assume that $g = x_0 c_1 x_1 \dots c_n x_n$, where $x_i \in A$ for $i = 0,1,\dots,n$ and $c_i \in C$ for $i = 1, 2, \dots, n$, $f = y_0 d_1 y_1 \dots d_m y_m$, where $y_i \in A$ for $i = 0, 1, \dots, m$ and $d_i \in C$ for $i= 1,2, \dots, m$, are both in reduced form and $f = g$. We will proceed by induction on $m+n$.
	
	In case $m + n = 2$ we see that $m = n = 1$ and thus $g = x_0 c_1 x_1$, $f = y_0 d_1 y_1$. By the assumption we have that $y_1^{-1} d_1^{-1} y_0^{-1} x_0 c_1 x_1 = 1$. This product cannot be in a reduced form and thus $y_0^{-1} x_0 \in H$. Then we see that
	\begin{displaymath}
		y_1^{-1} d_1^{-1} y_0^{-1} x_0 c_1 x_1 = y_1^{-1} (d_1^{-1} c_1) (y_0^{-1} x_0 x_1) = 1.
	\end{displaymath}
	Again, this product is not reduced and thus $d_1 = c_1$ and we are done.
	
	Now assume that $m+n = K > 2$. Then clearly
	\begin{equation}
	\label{nf}
		y_m^{-1} d_m^{-1} \dots y_1^{-1} d^{-1}_1 y_0^{-1} x_0 c_1 x_1 \dots c_n x_n = 1
	\end{equation}
	and thus the left hand side of (\ref{nf}) is not reduced. Since both $f,g$ were in reduced form we see that $y_0^{-1} x_0 \in H$ and therefore $d^{-1}_1 y_0^{-1} x_0 c_1 = d^{-1}_1 c_1 y_0^{-1} x_0$ and thus we can rewrite (\ref{nf}) to
	\begin{equation*}
		y_m^{-1} d_m^{-1} \dots y_1^{-1} d^{-1}_1 c_1 y_0^{-1} x_0 x_1 \dots c_n x_n = 1.
	\end{equation*}
	Without loss of generality we may assume that $n \geq m$ and thus $x_1 \not\in H$. Since $x_1 \not\in H$ we see that $y_0^{-1} x_0 x_1 \not\in H$ and thus $d^{-1} c_1 = 1$ and we can rewrite (\ref{nf}) to
	\begin{equation*}
		y_m^{-1} d_m^{-1} \dots y_1^{-1} y_0^{-1} x_0 x_1 \dots c_n x_n = 1.
	\end{equation*}
	Since both $f$ and $g$ were in reduced form we see that $f_1 = g_1$, where $f_1 = (x_0 x_1) c_2 x_2 \dots c_n x_n$ and $g_1 = (y_0 y_1) d_2 y_2 \dots y_m d_m$ are in reduced form as well. Thus by induction hypothesis we get that $m = n$ and $d_i = c_i$ for $i = 2, \dots, n$.
\end{proof}

The above lemma shows that if $g = x_0 c_1 x_1 \dots c_n x_n$ is reduced then $c_1, \dots, c_n$ are given uniquely. We will call them the \emph{consonants} of $g$. Denote $|g|_C = n$ and we will call $|g|_C$ the \emph{consonant length} of $g$. 

Special amalgams are useful because they have a functorial property.

\begin{remark}
\label{extension}
Let $H,A, C, Q, S$ be groups such that $H \leq A$ and let $\psi_A \colon A \to B$, $\psi_C \colon C \to S$ be group homomorphisms. Then by universal property of amalgamated free products $\psi_A, \psi_C$ uniquely extend to a homomorphism $\psi: G \rightarrow P$, where $G = A\specam_{H}C$ and $P = Q \specam_{\psi(H)}S$, such that
\begin{displaymath}
	\psi(g) =	\left\{\begin{array}{ll}
					\psi_A(a)	& \mbox{ if } g = a \mbox{ for some } a\in A,\\
					\psi_C(c)	& \mbox{ if } g = c \mbox{ for some } c\in C.
			 	\end{array}\right.
\end{displaymath}
\end{remark}

\begin{lemma}
	\label{extension_kernel}
	With notation as stated in Remark \ref{extension}, $\ker(\psi) = \lnorm \ker(\psi_A),\ker(\psi_C) \rnorm^G$.
\end{lemma}
\begin{proof}
Let's use $N = \lnorm \ker(\psi_A), \ker(\psi_C) \rnorm^G$. Obviously $N \leq \ker(\psi)$, thus we need to show the opposite inclusion.
	
Let $\phi: G \rightarrow G/N$ be the natural projection, thus $N = \ker(\phi)$. Let $\theta: G/N \rightarrow P$ be a homomorhism such that $\ker(\theta)=\phi(\ker(\psi))$ and $\psi=\theta \circ \phi$. Note that $\ker(\psi_A) = \ker(\psi) \cap A$ and $\ker(\psi) \cap C = \ker(\psi_C)$ thus it makes sense to define $\xi_A \colon \psi(A) \to \phi(A)$ be the homomorphism given by $\xi_A(\psi(a))=\phi(a)$ and $\xi_C \colon \psi(C) \to \phi(C)$ be the homomorphism given by $\xi_C(\psi(c))=\phi(c)$. Clearly $\xi_A, \xi_C$ are isomorphisms. Let $h \in H$ and $c\in C$, then $[\xi_A(\psi(h)), \xi_C(\psi(c))] = [\phi(h),\phi(c)] = \phi([h,c]) = 1$. Therefore by von Dyck's theorem the homomorphisms $\xi_A, \xi_C$ extend to a homomorphism $\xi \colon P \to G/L$ defined on the generators of $P$ by
	\begin{displaymath}
		\xi(p) =
			\left\{ \begin{array}{ll}
				\phi(a) &\mbox{ if }p = \psi_A(a) \mbox{ for some } a \in A,\\
				\phi(c) &\mbox{ if }p = \psi_C(c) \mbox{ for some } c \in C.
			\end{array}\right.
	\end{displaymath}

Then clearly $\xi\circ\theta \colon G/L \to G/L$ is the identity as it is defined on the generators of $G/L$ by following:
\begin{displaymath}
		\xi\circ \theta(q) = 
			\left\{ \begin{array}{ll}
				\xi(\theta(\phi(a))) = \xi(\psi(a)) = \phi(a) &\mbox{ if }q = \phi(a) \mbox{ for some } a \in A,\\
				\xi(\theta(\phi(c))) = \xi(\psi(c)) = \phi(c) &\mbox{ if }q = \phi(c) \mbox{ for some } c \in C.
			\end{array}\right.
	\end{displaymath}
It follows that $\xi$ is injective and thus $\ker(\theta) = \{1\} = \phi(\ker(\psi))$. Therefore $\ker(\psi) \leq \ker(\phi)$. Altogether we see that $\ker(\psi) = \lnorm \ker(\psi_A), \ker(\psi_C) \rnorm^G$.
\end{proof}

\subsection{Cyclically reduced elements and conjugacy}
From now on let $H \leq A, C$ be groups and let $G$ denote $A \specam_H C$, the special amalgam of $A$ and $C$ along $H$.

\begin{definition}
Let $g = c_1 x_1 \dots c_n x_n$, where $x_i \in A$ and $c_i \in C$ for $i=1, \dots,n$. We say that $g$ is \emph{cyclically reduced} if $c_1 x_1 \dots c_n x_n$ is a reduced expression and if $n \geq 2$ then $x_n \not\in H$. We will say that an element $p \in G$ is a \emph{prefix} of $g$ if $p = c_1 x_1 \dots c_l x_l$ for some $0 \leq l \leq n$ and that $s \in G$ is a \emph{suffix} of $g$ if $s = c_{n-m} x_{n-m} \dots c_n x_n$ for some $-1 \leq m \leq n-1$. 
\end{definition}
Note that we define prefix and suffix only for cyclically reduced elements. 

\begin{lemma}
\label{conjugator_in_special_amalgam}
	Let $g = c_1 x_1 \dots c_n x_n$ and $f = d_1 y_1 \dots d_n y_n$, where $x_i,y_i \in A$ and $c_i, d_i \in C$ for $i = 1, 2, \dots, n$, be cyclically reduced elements of $G$ such that $n \geq 1$ and $x_n \not\in H$. Assume that $ugu^{-1}=f$ for some $u \in G$. Let $u = z_0e_1z_1 \dots e_m z_m$, where $z_i \in A$ and $e_j \in C$ for $i = 0, 1, \dots, m$ and $j = 1, \dots, m$, be a reduced expression. Then exactly one of the following is true
	\begin{enumerate}
		\item[a)] $m = 0$ and $u \in H$,
		\item[b)] $m \geq 1$, $z_m \in H$ and there is a prefix $p$ of $g$ such that $u = hp^{-1}g^{-l}$ for some $h \in H$ and $l \in \mathbb{N}_0$,
		\item[c)] $m \geq 1$, $x_n z_m^{-1} \in H$ and there is a suffix $s$ of $g$ such that $u = hsg^l$ for some $h \in H$ and $l \in \mathbb{N}_0$.
	\end{enumerate}
\end{lemma}
\begin{proof}
	If $m=0$, then $u = z_0$ and thus
	\begin{displaymath}
		y_n^{-1}d_n^{-1} \dots y_1^{-1}d_1^{-1}z_0 c_1 x_1 \dots c_n x_n z_0^{-1} = 1.
	\end{displaymath}
	This product clearly cannot be reduced and therefore $z_0$ must belong to $H$.
	
	Now suppose $m \geq 1$. Then since $f = ugu^{-1}$ we get
	\begin{displaymath}
		z_0 e_1 z_1 \dots e_m z_m c_1 x_1 \dots c_n x_n z_m^{-1} e_m^{-1} \dots z_1^{-1} e_1^{-1} z_0^{-1} = d_1 y_1 \dots d_n y_n.
	\end{displaymath}
	Right hand side of this equation is shorter than left hand side and right hand side is reduced by assumption, therefore left hand side cannot be reduced and thus we see that either $z_m \in H$ or $x_n z_m^{-1} \in H$. Since $x_n \not\in H$ we see that exactly one of these two possibilities may happen.
	
	First suppose that $z_m \in H$. Then $e_m z_m = z_m e_m$ and thus we have
	\begin{equation}
	\label{rovnice_1}
		z_0 e_1 z_1 \dots e_{m-1}(z_{m-1}z_m) (e_m c_1) x_1 \dots c_n x_n z_m^{-1} e_m^{-1} \dots z_1^{-1} e_1^{-1} z_0^{-1} = d_1 y_1 \dots d_n y_n.
	\end{equation}
	Again, left hand side cannot be reduced. Since $z_m \in H$ we cannot have $z_{m-1}z_m \in H$ as that would make $z_{m-1}\in H$ which would contradict our assumption that $u$ is reduced. Thus we must have that $e_m = c_1^{-1}$. Denote $h_1 := z_{m-1}z_{m}x_1 \in A$. Then
	\begin{displaymath}
		z_m^{-1} e_m^{-1} z_{m-1}^{-1} = e_m^{-1}z_m^{-1}z_{m-1}^{-1} = c_1 z_m^{-1} z_{m-1}^{-1} = c_1 x_1 x_1^{-1} z_m^{-1} z_{m-1}^{-1} = c_1 x_1 h_1^{-1}
	\end{displaymath}
	Since $z_{m-1} z_m = h_1 x_1^{-1}$ we get 
	\begin{displaymath}
		(z_0 e_1 z_1 \dots e_{m-1} h_1)(c_2 x_2 \dots c_n x_n c_1 x_1)(z_0 e_1 z_1 \dots e_{m-1} h_1)^{-1} = d_1 y_1 \dots d_n y_n.
	\end{displaymath}
	Set $u' = z_0 e_1 z_1 \dots e_{m-1} h_1$ and $g'= c_2 x_2 \dots c_n x_n c_1 x_1$.
	On the left hand side of the equation (\ref{rovnice_1}) we have
	\begin{displaymath}
		z_0 e_1 z_1 \dots e_{m-1} h_1 c_2 x_2 \dots c_n x_n c_1 x_1 h^{-1}_1 e_{m-1}^{-1} \dots z_1^{-1} e_1^{-1} z_0^{-1}
	\end{displaymath}
	and since this expression has longer consonant length than the right hand side of the equation we see that it cannot be reduced and therefore $h_1 = z_{m-1} z_m x_1 \in H$.
	If $m=1$ we get that $u = h_1 (c_1 x_1)^{-1}$ and the lemma is proved.
	
	Now suppose that $m = M > 1$ and that the statement has been already proved for all $u \in G$ such that $|u|_C < M$, thus we can use the induction hypothesis for $f, g'$ and $u'$ as $|u'|_C = |u|_C - 1$.
	
	We have $u' g' {u'}^{-1} = f$, 
	\begin{equation}
	\label{fin}
		z_0 e_1 z_1 \dots e_{m-1} h_1 c_2 x_2 \dots c_n x_n c_1 x_1 h^{-1}_1 e_{m-1}^{-1} \dots z_1^{-1} e_1^{-1} z_0^{-1} = d_1 y_1 \dots d_n y_n
	\end{equation}
	Since we have already shown that $h_1 \in H$ we can use induction hypothesis and by $b)$ we see that there is a prefix $p'$ of $g'$ such that $z_0 e_1 z_1 \dots z_{m-2} e_{m-1} h_1 = h {p'}^{-1}{g'}^{-l}$ for some $h\in H$ and $l \in \mathbb{N}$. As a result of this we have
		\begin{align*}
			u 	&= z_0 e_1 z_1 \dots e_{m-1} z_{m-1} e_m z_m\\
				&= z_0 e_1 z_1 \dots e_{m-1} z_{m-1} z_{m} c_1^{-1}\\
				&= z_0 e_1 z_1 \dots e_{m-1} h_1 x_1^{-1} c_1^{-1}\\
				&= h{p'}^{-l} {g'}^{-l} x_1^{-1} c_1^{-1} = h {p'}^{-1}x_1^{-1}c_1^{-1}g^{-l}= h (c_1 x_1 p')^{-1} g^{-l}.
		\end{align*}
Now two possibilities can occur. As $p'$ is a prefix of $g'$ we see that $p' = c_2 x_2 \dots c_k x_k$ where $2 \leq k \neq n$. Then either $c_1 x_1 p'$ is a prefix of $g$ or $c_1 x_1 p' = g$. Either way we are done.

In case $x_n z_m^{-1} \in H$ we can proceed analogously.
\end{proof}

Let $g = c_1 x_1 \dots c_n x_n$, where $x_1, \dots, x_n \in A$ and $c_1, \dots, c_n \in C$ with $h$, be a cyclically reduced element of $G$. Then we say that $g' \in G$ is a \emph{cyclic permutation} of $g$ if $g' = c_m x_m \dots c_n x_n c_1 x_1 \dots c_{m-1} x_{m-1}$ for some $1 \leq m \leq n$. Equivalently, $g'$ is a cyclic permutation of $g$ if there is $f$, a prefix (or a suffix) of $g$ , such that $g = x^{-1}gx$ (or $g' = xgx^{-1}$).

\subsection{Centralisers and a conjugacy criterion}
Recall that $G = A \specam_H C$. The following lemma is a special version of \cite[Chapter IV, Theorem 2.8]{shupp}.
\begin{lemma}[Conjugacy criterion for special amalgams]
	\label{collinss_lemma}
	Let $g = c_1 x_1 \dots c_n x_n$ and $f = d_1 y_1 \dots d_m y_m$, where $x_1, \dots, x_n, y_1, \dots y_m \in A$, and $c_1, \dots, c_n, d_1, \dots, d_m \in C$, be cyclically reduced elements of $G$ with $n \geq 1$. Then $g \not \in A^G$. If $f \sim_G g$ then $m = n$ and there is $g' \in G$, a cyclic permutation of $g$, such that $f \sim_H g'$.
\end{lemma}
	Clearly every cyclically reduced element has only finitely many cyclic permutations. Lemma \ref{collinss_lemma} motivates us to give a sufficient and necessary condition for whether two cyclically reduced elements of $G$ are conjugate by some element of $H$.
	
\begin{lemma}
	Suppose $g = cx \in G$ such that $c \in C\setminus\{1\}$ and $x \in H$. Then 
	\begin{displaymath}
		C_G(g)=C_C(c) C_H(x) \cong C_C(c) \times C_H(x).
	\end{displaymath}
\end{lemma}
\begin{proof}
	Obviously $f \in C_G(g)$ if and only if $fgf^{-1} = g$. Let $f \in C_G(g)$ and let $f =  z_0 e_1 z_1 \dots e_m z_m$, where $z_0, z_1, \dots, z_m \in A$ and $e_1, \dots, e_m \in C$, be the reduced expression for $f$. Then
	\begin{equation}
		\label{rovnice}
		z_0 e_1 z_1 \dots e_m z_m c x z_m^{-1} e_m^{-1} \dots z_1^{-1} e_1^{-1} z_0^{-1} = cx.
	\end{equation}
	If $m=0$ we get that $z_0 c x z_0^{-1} x^{-1} c^{-1} = 1$ thus $x z_0^{-1} x^{-1} \in H$, therefore $z_0 \in H$ and $z_0 c = c z_0$. Consequently we get $z_0 x z_0^{-1}x^{-1} = 1$ and thus $z_0 \in C_H(x)$.
	
	Assume $m\geq 1$. Then either $x z_m^{-1} \in H$ or $z_m \in H$. Since $x \in H$ we see that both must be true, thus $z_m c = c z_m$ and $e_m z_m c x  z_m^{-1}e_m^{-1} = (e_m c e_m^{-1})(z_m x z_m^{-1})$ and thus (\ref{rovnice}) rewrites to 
\begin{displaymath}
		z_0 e_1 z_1 \dots e_{m-1} z_{m-1} (e_m c e_m^{-1})(z_m x z_m^{-1}) z_{m-1}^{-1} e_{m-1}^{-1} \dots z_1^{-1} e_1^{-1} z_0^{-1} = cx.
\end{displaymath}
Since $c \neq 1$ we see that $e_m c e_m^{-1} \neq 1$. Left hand side cannot be reduced and therefore either $z_{m-1} \in H$ or $(z_m x z_m^{-1})z_{m-1}^{-1} \in H$. Since $z_m x z_m^{-1} \in H$ we see that both must be true. If $m \geq 2$ this contradicts $z_{m-1} \not\in H$.

Thus we may assume that $m=1$ and consequently $f = z_0 e z_1$ with $z_0, z_1 \in H$. Therefore
\begin{align*}
	g^{-1}fgf^{-1} 	&= x^{-1} c^{-1} z_0 e z_1 c x z_1^{-1} e^{-1} z_0^{-1}\\
					&= x^{-1} (c^{-1} e c) (z_0 z_1 x z_{1}^{-1}) e^{-1} z_0^{-1} = 1.
\end{align*}
Since $m \geq 1$ we see that $e \neq 1$ and consequently $c^{-1} e c \neq 1$. This leaves us with $z_0 z_1 x z_{1}^{-1} \in H$. As $z_1 x z_1^{-1} \in H$ and $z_0 \in H$ we see that
\begin{displaymath}
	1 = g^{-1}fgf^{-1} = (c^{-1}ece^{-1})\cdot(x^{-1}(z_0 z_1) x (z_0 z_1)^{-1}) \in C H.
\end{displaymath}
This gives us that $e \in C_C(c)$ and $z_0 z_1 \in C_H(x)$. Altogether we see that $f = e z_0 z_1 \in C_C(c)\times C_H(x)$.
\end{proof}

\begin{lemma}
\label{7.8}
	Let $H \leq A, C$ be groups and let $G = A\specam_H C$. Suppose that $g = c_1 x_1 \dots c_n x_n$, $c_i \in C$, $x_i \in A$ for $i = 1, \dots, n$, is cyclically reduced in $G$ and $n \geq 1$.
	
	If $x_n \in H$ then $n=1$ and $C_G(g) = C_C(c_1) \times C_H(x_1) \leq G$.
	
	If $x_n \not\in H$, let $\{p_1, \dots ,p_k\}$, where $1 \leq k \leq n+1$, be the set of all prefixes of $g$ satisfying $p_i^{-1} g p_i \in g^H$. For each $i = 1, \dots, k$ choose $h_i \in H$ such that $h_i p_i^{-1} g p_i h_i^{-1} = g$ and define finite subset $\Omega \subseteq G$ by $\Omega = \{h_i p_i^{-1} | i=1,\dots,k\}$.
	
	Then $C_G(g) = C_H(g)\langle g \rangle \Omega$.
\end{lemma}
\begin{proof}
	If $x_n \in H$ then $g$ is cyclically reduced in $G$ if and only if $n=1$. Then the claim follows from the previous lemma.
	Suppose $x_n \in A \setminus H$. Let $u \in C_G(g)$ thus $g = ugu^{-1}$. Then by Lemma \ref{conjugator_in_special_amalgam} we know that there are $h\in H$ and $l \in \mathbb{N}$ such that either there is a prefix $p$ of $g$ such that $u = hp^{-1} g^{-l}$ or there is a suffix $s$ of $g$ such that $u =  hsg^l$. In the second case there is a prefix $p$ of $g$ such that $s = p^{-1}g$ and thus $u = hsg^l = h p^{-1} g^{l+1}$. Thus without loss of generality we may assume that $u = hp^{-1} g^l$ for some prefix $p$ of $g$, $h\in H$ and $l\in \mathbb{Z}$.
	
	We see that $g = hp^{-1}g^l g g ^{-l} p h^{-1} = h p^{-1} g p h^{-1}$, therefore $p \in \{p_1, \dots, p_k\}$ and thus there is $h_i \in H$ such that $h_i p^{-1} g ph_i^{-1} = g = h p^{-1} g p h^{-1}$. This yields $hh_i^{-1} \in C_H(g)$, thus 
	\begin{displaymath}
		u = hp_i^{-1}g^l \in C_H(g)h_i p_i^{-1} g^l \subseteq C_H(g) \Omega \langle g \rangle.
	\end{displaymath}
	Since $\Omega \subseteq C_G(g)$ we see that $C_H(g) \Omega \langle g \rangle = C_H(g) \langle g \rangle \Omega$.
	
	So it has been proven that $C_G(g) \subseteq C_H(g) \langle g \rangle \Omega$. Inclusion in the opposite direction is obvious.
\end{proof}

\begin{lemma}
	\label{conjugacy_criterion_special_amalgams}
	Let $H \leq A, C$ be groups and let $G = A\specam_H C$. Suppose $B \leq A$ $f,g \in G$ are arbitrary. Let $g = x_0 c_1 x_1 \dots c_n x_n$, $f = y_0 d_1 y_1 \dots y_m d_m$, where $x_0, \dots, x_n, y_0, \dots, y_m \in A$ and $c_1, \dots, c_n, d_1, \dots, d_m \in C$, be reduced expressions for $g$ and $f$ respectively and assume that $n \geq 1$. Then $f \in g^B$ if and only if all of the following conditions are met
	\begin{enumerate}
		\item[(i)] $m = n$ and $c_i = d_i$ for $i = 1, \dots , n$,
		\item[(ii)] $y_0 y_1 \dots y_n \in (x_0 x_1 \dots x_n)^B$ in $A$,
		\item[(iii)] for every $b_0 \in B$ such that $y_0 y_1 \dots y_n = b_0 (x_0 x_1 \dots x_n) b_0^{-1}$ we have $I \neq \emptyset$ where
			\begin{displaymath}
				I = b_0 C_B(x_0 x_1 \dots x_n) \cap y_0 H x_0^{-1} \cap (y_0 y_1) H (x_0 x_1)^{-1} \cap \dots \cap (y_0 y_1 \dots y_{n-1}) H (x_0 x_1 \dots x_{n-1})^{-1}. 
			\end{displaymath}
	\end{enumerate}
\end{lemma}
\begin{proof}
	Suppose (i) - (iii) hold. Let $b_0 \in B$ be such that $y_0 y_1 \dots y_n = b_0 (x_0 x_1 \dots x_n) b_0^{-1}$ and let $b \in I$. Clearly $y_0 y_1 \dots y_n = b (x_0 x_1 \dots x_n) b^{-1}$ as $b \in b_0 C_B(x_0 x_1 \dots x_n)$. We want to show that $f^{-1} b g b^{-1} = 1$.
	Since (i) holds we can write
	\begin{equation}
		\label{prepisovana}
		f^{-1} b g b^{-1} = y_n^{-1} c_n^{-1} \dots y_1^{-1} c_1^{-1} y_0^{-1} b x_0 c_1 x_1 \dots c_n x_n b^{-1}.
	\end{equation}
	Since $b \in y_0 H x_0^{-1}$ we see that $y_0^{-1} b x_0 \in H$ and thus
	\begin{displaymath}
		y_1^{-1} c_1^{-1} y_0^{-1} b x_0 c_1 x_1 = y_1^{-1} y_0^{-1} b x_0 x_1,
	\end{displaymath}
	therefore we can rewrite (\ref{prepisovana}) to
	\begin{displaymath}
		f^{-1} b g b^{-1} = y_n^{-1} c_n^{-1} \dots y_2^{-1} c_2 (y_0 y_1)^{-1} b (x_0 x_1) c_2 x_2 \dots c_n x_n b^{-1}.
	\end{displaymath}
	Again, since $b \in (y_0 y_1) H (x_0 x_1)^{-1}$ we see that $(y_0 y_1)^{-1} b (x_0 x_1) \in H$ and thus (\ref{prepisovana}) rewrites to
	\begin{displaymath}
		f^{-1} b g b^{-1} = y_n^{-1} c_n^{-1} \dots y_3^{-1} c_3 (y_0 y_1 y_2)^{-1} b (x_0 x_1 x_2) c_3 x_3 \dots c_n x_n b^{-1}.
	\end{displaymath}
	By repeating this argument $n$-times we rewrite (\ref{prepisovana}) to
	\begin{displaymath}
		f^{-1} b g b^{-1} = (y_0 y_1 \dots y_n)^{-1} b (x_0 x_1 \dots x_n) b^{-1}.
	\end{displaymath}
	Which is equal to $1$ by assumption, thus $f \in g^B$.
	
	Now assume $f\in g^B$, so there is $b \in B$ such that
	\begin{displaymath}
		y_0 d_1 y_1 \dots d_m y_m = b (x_0 c_1 x_1 \dots c_n x_n) b^{-1}.
	\end{displaymath}
	By Lemma \ref{nft_special_amalgam} we see that $m = n$ and $c_i = d_i$ for $i = 1, \dots, n$, thus we've established (i).
	There is a natural retraction $\rho: G \rightarrow A$ defined by $\rho(a) = a$ for all $a \in A$ and $\rho(c) = 1$ for all $c \in C$. By applying this retraction we establish (ii). Let $b_0$ be an arbitrary element of $B$ such that $y_0 y_1 \dots y_n = b_0 (x_0 x_1 \dots x_n) b_0^{-1}$. Obviously $b \in b_0 C_B(x_0 x_1 \dots x_n)$. From assumptions we have that
	\begin{equation}
		\label{prepisovana2}
		y_n^{-1}c_n^{-1} \dots y_1^{-1} c_1^{-1} (y_0^{-1}b x_0)c_1 x_1 \dots c_n x_n b^{-1} = 1 
	\end{equation}
	By Lemma \ref{nft_special_amalgam} this expression cannot be reduced thus $y_0^{-1} b x_0 \in H$, therefore $b \in y_0 H x_0^{-1}$. This means that we can rewrite (\ref{prepisovana2}) as
	\begin{displaymath}
		y_n^{-1}c_n^{-1} \dots y_2^{-1} c_2^{-1} (y_1 ^{-1} y_0^{-1}b x_0 x_1)c_2 x_2 \dots c_n x_n b^{-1} = 1. 
	\end{displaymath}
	Again, by Lemma \ref{nft_special_amalgam} we see that $(y_0 y_1)^{-1}b (x_0 x_1) \in H$ or equivalently $b \in (y_0 y_1)H (x_0 x_1)^{-1}$. By applying this step $n$-times we establish (iii), thus $I \neq \emptyset$.
\end{proof}

\begin{lemma}
	\label{special_amalgam_centraliser}
	Let $H \leq A, C$ be groups and let $G = A\specam_H C$, suppose $B \leq A$. Let $g \in G$ and $g = x_0 c_1 x_1 \dots c_n x_n$, where $x_0, \dots, x_n \in A$ and $c_1, \dots, c_n \in C$, be a reduced expression of $g$ with $n \geq 1$. Then $C_B(g) = I$, where
	\begin{displaymath}
				I= C_B(x_0 x_1 \dots x_n) \cap x_0 H x_0^{-1} \cap (x_0 x_1) H (x_0 x_1)^{-1} \cap \dots \cap (x_0 x_1 \dots x_{n-1}) H (x_0 x_1 \dots x_{n-1})^{-1}. 
	\end{displaymath}
\end{lemma}
\begin{proof}
	Clearly $g \sim_G g$ and thus by previous lemma for any $b_0 \in B$ such that $b_0 (x_0 x_1 \dots x_n) b_0^{-1} = x_0 x_1 \dots x_n$ we have $I \neq \emptyset$, where
	\begin{displaymath}
	I = b_0C_B(x_0 x_1 \dots x_n) \cap x_0 H x_0^{-1} \cap (x_0 x_1) H (x_0 x_1)^{-1} \cap \dots \cap (x_0 x_1 \dots x_{n-1}).
	\end{displaymath}
	We can set $b_0 = 1$. Now take $b \in I$ by argumentation analogous to the proof of the previous lemma we see that $bgb^{-1}=g$ and thus $I \subseteq C_B(g)$.
	
	Let $b\in C_B(g)$. Then $bgb^{-1} = g$. By the previous lemma we see that $b \in I$. Therefore $I = C_G(g)$.
\end{proof}

%% file: hcs.tex
\section{Finite graph products of $\C$-HCS groups}
\label{HCS}

From now on we will assume that the class $\C$ is an extension closed variety of finite groups, i.e. $\C$ satisfies (c1), (c2), (c3) and (c4). The main result of this section is the following generalisation of \cite[Theorem 1.1]{raags} and \cite[Theorem 6.15]{toinet}.
\begin{thrm}
	\label{main_hcs}
	Assume that $\C$ is an extension closed variety of finite groups. Then the class of $\CHCS$ groups is closed under taking finite graph products.
\end{thrm}

\subsection{Some auxiliary statements}

The following two statements were proved first by Minasyan in \cite[Lemma 5.6 and 5.7]{raags} in case when $\mathcal{C}$ is the class of all finite groups. Later in his paper \cite{toinet} Toinet proved them in case when $\C$ is the class of all finite $p$-groups for some prime number $p$. The proofs can easily be generalised for the case when the class $\C$ is an extension closed variety of finite groups and we leave them to the reader.

\begin{lemma}
\label{ashot_5.6}
Let $G$ be a group and let $A, B \leq G$ be retracts of $G$ with corresponding retractions $\rho_A, \rho_B \in \End(G)$ such that $\rho_A \circ \rho_B = \rho_B \circ \rho_A$. Let $x$ be an element of $G$ and let $\alpha = \rho_A(\rho_B(x)x^{-1})x \rho_B(x^{-1}) \in G$. Suppose that the pair $(A \cap B, \alpha)$ satisfies $\mathcal{C}\text{-CC}$ in $G$. Then for any $K \in \mathcal{N_C}(G)$ there exists $M \in \mathcal{N_C}(G)$ such that $M \leq K$, $\rho_A(M) \subseteq M$, $\rho_B(M) \subseteq M$ and $\phi(A) \cap \phi(xBx^{-1}) \subseteq \phi(A \cap xBx^{-1})\phi(K)$ in $G/M$, where $\phi \colon G \to G/M$ is the natural epimorphism.
\end{lemma}

\begin{lemma}
\label{ashot_5.7}
Let $G$ be a group and let $A,B \leq G$ be retracts of $G$ with corresponding retractions $\rho_A, \rho_B \in \End(G)$ such that $\rho_A \circ \rho_B = \rho_B \circ \rho_A$. Consider arbitrary elements $x,g \in G$. Denote $D = xBx^{-1} \leq G$ and $\alpha = \rho_A(\rho_B(x)x^{-1})x \rho_B(x^{-1}) \in G$. Suppose that the conjugacy classes $\alpha^{A\cap B}$ and $g^{A\cap D}$ are $\mathcal{C}$-closed in $G$, and the pair $(A \cap B, \alpha)$ satisfies $\mathcal{C}\text{-CC}_G$. Then the double coset $C_A(g)D$ is $\mathcal{C}$-closed in $G$.
\end{lemma}

Dyer \cite[Theorem 3]{dyer} proved that free-by-finite groups are CS. In \cite[Theorem 4.2]{toinet} Toinet proved that free-by-(finite-$p$) groups are $p$-CS. Ribes and Zalesskii generalised these results (see \cite[Section 3, Theorem 3.2]{zalesskii}) to the following.
\begin{thrm}
	\label{zalesskii}
	Let $\C$ be an extension closed variety of finite groups and let $G$ be finitely generated free-by-$\mathcal{C}$ group. Then $G$ is $\mathcal{C}$-CS.
\end{thrm}
Clearly, every $\C$-open subgroup of a finitely generated free-by-$\C$ group is finitely generated free-by-$\C$ group as well. We can state the following corollary as an immediate consequence of Theorem \ref{zalesskii}.
\begin{cor}
	\label{rz}
	Let $G$ be finitely generated free-by-$\C$ group. Then $G$ is $\mathcal{C}$-HCS.
\end{cor}

The following simple lemma will be crucial for our proofs.
\begin{lemma}
	\label{free-by-C}
	Suppose that $\C$ is a class of groups satisfying (c2). Let $Q, S \in \mathcal{C}$ and suppose that $R \leq Q$. Then $G = Q\specam_R S$ is free-by-$\C$.
\end{lemma}
\begin{proof}
	Let $\sigma \colon G \to Q\times S$ be the epimorphism defined on the generators of $G$ as follows:
		\begin{align*}
			\sigma(q) = q &\mbox{ for all $q \in Q$},\\
			\sigma(s) = s &\mbox{ for all $s \in S$}.
		\end{align*}
Clearly $\ker(\sigma) \in \mathcal{N_C}(G)$ as $\C$ is closed under taking direct products. We want to show that $\ker(\sigma)$ is a free group.  From the definition of $\sigma$ we see that $\ker(\sigma) \cap R \times S = \{1\}$. Let $T$ be the Bass-Serre tree for $Q \specam_R S = Q \ast_R (R \times S)$ and consider the induced action of $\ker(\sigma)$ on $T$. By a standard result of Bass-Serre theory (see \cite[Theorem 12.1]{bogopolski}) we know that the stabiliser of a vertex $v$ has to be conjugate either into $Q$ or $R\times S$, but since $\ker(\sigma)$ is normal and does not intersect any of the factors we see that $\ker(\sigma)$ acts freely on $T$ and thus it is free. As a consequence we see that $G$ is free-by-$\mathcal{C}$.
\end{proof}

Combining corollary \ref{rz} together with Lemma \ref{free-by-C} we immediately get the following.
\begin{cor}
	\label{my_dyer}
	Suppose that $\C$ is an extension closed variety of finite groups. Let $Q, S \in \mathcal{C}$ and assume that $R \leq Q$. Then $G = Q\specam_R S$ is $\mathcal{C}$-HCS. 
\end{cor}

\subsection{Proof of Theorem {\protect\mainhcsref}}
Before we proceed to the proof we first prove a weaker statement. This was first proved by Green in \cite{green} both for the case when $\mathcal{C}$ is the class of all finite groups and for the case when $\mathcal{C}$ is the class of all finite $p$-groups for some prime number $p$.
\begin{lemma}
	\label{GP_residually_P}
	Let $\mathcal{C}$ be a class of finite groups satisfying (c1), (c2) and (c4). Then the class of residually-$\C$ groups is closed under taking graph products.
\end{lemma}
\begin{proof}
	First, we show the statement holds for all finite graph products. The proof will by done by induction on $|V\Gamma|$. If $|V\Gamma| = 1$ we see that $\Gamma \mathcal{G} = G_v$ and $G_v$ is residually-$\mathcal{C}$ by assumption.
	
	Now assume that the statement has been proved for all graph products $\Gamma\mathcal{G}$ such that $|V\Gamma| \leq r$. Let $G = \Gamma \mathcal{G}$ be such that $|V\Gamma| = r+1$	and let $g \in G \setminus \{1\}$ be arbitrary. Pick $v \in V\Gamma$ and denote $C = G_v$, $H = G_{\link_{\Gamma}(v)}$ and $A = G_{V\Gamma \setminus \{v\}}$. Then clearly $G = A \specam_H C$ by Remark \ref{structure_of_graph_products}. By the induction hypothesis we get that $A, C, H$ are residually-$\mathcal{C}$. Let $g = x_0 c_1 x_1 \dots c_n x_n$, where $x_0, x_1, \dots, x_n \in A$ and $c_1, \dots, c_n \in C$, be a reduced expression for $g$ in $G$.
	There are two cases to consider: either $n = 0$ or $n \geq 1$. If $n = 0$ then $g = x_0 \in A\setminus\{1\}$ and we can use the fact that $A$ is a retract in $G$, thus we can consider the canonical retraction $\rho_A \colon G \to A$. Then $\rho_A(x_0) = x_0$ and $A$ is residually-$\mathcal{C}$ by induction hypothesis.
	
 Suppose that $n \geq 1$. Clearly, $H$ is a retract of $A$ and therefore by Lemma \ref{retract_restriction} we see that $H$ is $\mathcal{C}$-closed in $A$. This means that there is a group $Q \in \mathcal{C}$ and an epimorphism $\alpha \colon A \to Q$ such that $\alpha(x_i) \not \in \alpha(H)$ for whenever $x_i \not \in H$. Similarly since $C$ is residually-$\mathcal{C}$ by assumption as it is a vertex group we see that there is a group $S \in \mathcal{C}$ and an epimorphism $\gamma \colon C \to S$ such that $\gamma(c_i) \neq 1$ in $S$ for all $i = 1, \dots, n$. Let $\phi \colon G \to P$, where $P = Q \specam_{\alpha(H)}S$, be the canonical extension of $\alpha$ and $\gamma$ (see Remark \ref{extension}). We see that
	\begin{displaymath}
 		\phi(g) = \alpha(x_0) \gamma(c_1) \alpha(x_1) \dots \gamma(c_n) \alpha(x_n),
	\end{displaymath}
	is a reduced expression and thus $\phi(g)$ nontrivial in $P$ is by Lemma \ref{nft_special_amalgam}. By Lemma \ref{free-by-C} we see that $P$ is free-by-$\C$ and thus $P$ is residually-$\C$ by Lemma \ref{free-by-C_is_residually-C}.
	
	We have showed that both in case if $n = 0$ and if $n \geq 1$ we can separate $g$ from $\{1\}$. Using Lemma \ref{closed_simplification} we see that $\{1\}$ is $\mathcal{C}$-closed in $G$ and thus $G$ is residually-$\mathcal{C}$.
	
	Now, assume that the graph $\Gamma$ is infinite and let $g \in G \setminus \{1\}$ be arbitrary. Obviously, $S = \supp(g)$ is finite. Let $G_S$ be the full subgroup corresponding to $S$ and let $\rho_S \colon G \to G_S$ be the canonical retraction. Clearly, $\rho_S(g) = g \neq 1$ and $G_S$ is residually-$\C$ as it is a finite graph product of residually-$\C$ groups. Using Lemma \ref{closed_simplification} we se that $\{1\}$ is $\C$-closed in $G$ and thus $G$ is residually-$\C$.
\end{proof}

The main idea of the proof of Theorem \ref{main_hcs} is somewhat similar to the proof of Lemma \ref{GP_residually_P}. However, significantly more work needs to be done. To be able to prove Theorem \ref{main_hcs} we will need the following two lemmas.
\begin{lemma}
	\label{hlavni1}
	Let $\Gamma$ be a finite graph and let $G = \Gamma\mathcal{G}$ be a graph product where $G_v$ is $\mathcal{C}\text{-HCS}$ for all $v \in V\Gamma$. Then all full subgroups of $G$ satisfy $\mathcal{C}\text{-CC}_G$.
\end{lemma}
\begin{lemma}
\label{hlavni2}
	Let $\Gamma$ be a finite graph and let $G = \Gamma\mathcal{G}$ be a graph product where $G_v$ is $\mathcal{C}\text{-HCS}$ for all $v \in V\Gamma$. Then for all $g \in G$ and all full subgroups $B \leq G$ the set $g^B$ is $\mathcal{C}$-separable in $G$.
\end{lemma}

Lemmas \ref{hlavni1} and \ref{hlavni2} will be proved simultaneously by induction on $|V\Gamma|$. If $|V\Gamma| = 1$ we see that both lemmas hold trivially as $G = G_v$ which is $\mathcal{C}$-HCS by assumption.
Now assume that the two lemmas are true for all $\Gamma \mathcal{G}$ where $|V\Gamma| \leq r$.

To be able to control conjugacy classes and centralisers in special amalgam $A\specam_H C$ we need to be able to control intersections of conjugates of the amalgamated group $H$ inside $A$ as stated in Lemmas \ref{7.8} and \ref{conjugacy_criterion_special_amalgams}. In terms of our setting with graph products this means that we need to be able to control intersections of conjugates of full subgroups. This is established in Lemma \ref{humus}. The rest of Section \ref{HCS} is a case analysis dealing with all possible situations that might occur and shows that in all of the cases we can construct a suitable homomorphism from our graph product onto a special amalgam groups belonging to the class $\C$ which is a free-by-$\C$ group and thus by Corollary \ref{my_dyer} is $\C$-CS group. 

\begin{remark}
	\label{pomocna}
	Let $G$ be a group and let $H, F \leq G$, $b,x,y \in G$ be arbitrary. If $b H \cap xFy \neq \emptyset$ then for any $a \in b H \cap xFy$ we have $b H \cap xFy = a(H \cap y^{-1} F y)$. 
\end{remark}
\begin{proof}
	Let $a \in b H \cap xFy$. Since $ a \in bH$ we see that $aH = bH$. Since $a \in xFy$ we have $a = xfy$ for some $f \in F$. Thus we can write
	\begin{displaymath}
		bH\cap xFy = aH \cap aa^{-1}xFy = aH \cap ay^{-1}f^{-1}x^{-1}xFy = aH \cap ay^{-1}Fy = a(H \cap y^{-1}Fy).
	\end{displaymath}
\end{proof}

The following statements and their proofs very closely follow the contents of \cite[Section 8]{raags}.

\begin{lemma}
	\label{humus}
	Let $G$ be a graph product and assume that every full subgroup $B \leq G$ satisfies $\mathcal{C}\text{-CC}_G$ and for each $g \in G$ the conjugacy class $g^B$ is $\mathcal{C}$-closed in $G$.
	
	Let $A_1, \dots, A_n \leq G$ be full subgroups of $G$, let $A_0$ be a conjugate of a full subgroup of $G$ and let $b, x_0, x_i, y_i \in G$ for $i = 1, \dots, n$. Then for any $K \in \mathcal{N_C}(G)$ there is $L \in \mathcal{N_C}(G)$ such that $L \leq K$ and
	\begin{displaymath}
		\overline{b}C_{\overline{A}_0}(\overline{x}_0)\cap \bigcap_{i=1}^{n} \overline{x}_i\overline{A}_i \overline{y}_i \subseteq \psi \left(\left( b C_{A_0}(x_0) \cap \bigcap_{i=1}^n {x_i A_i y_i}\right) K\right) \mbox{ in } G/L,
	\end{displaymath}
	where $\psi: G \rightarrow G/L$ is the natural projection and $\overline{b} = \psi(b)$, $\overline{A}_i = \psi(A_i)$ $\overline{x}_i = \psi(x_i)$, $i = 0, \dots, n$ and $\overline{y}_j = \psi(y_j)$, $j = 1, \dots, n$.
\end{lemma}
\begin{proof}
	We will proceed by induction on $n$. If $n = 0$ then we just want $\overline{b}C_{\overline{A}_0}(\overline{x}_0) \subseteq \psi\left( b C_{A_0}(x_0)K\right)$. By assumption $A_0 = hAh^{-1}$ for some $h \in G$ and $A \leq G$ and thus the pair $(A, h^{-1} g h)$ has $\mathcal{C}\text{-CC}_G$. We can consider $\phi_{h^{-1}}$, the inner automorphism of $G$ given by $h^{-1}$. Obviously $C_{\phi_{h^{-1}}(H)}(\phi(g)) = \phi_{h^{-1}}(C_H(g)) \subseteq \phi_{h^{-1}}(C_H(g)K)$ in $G$ for any $K \in \mathcal{N_C}(G)$ since $\phi_{h^{-1}} \in \Aut(G)$. Since $\ker(\phi_{h^{-1}}) = \{1\}$ we can use Lemma \ref{cc_simplification} to see that the pair $(A_0, g)$ has $\mathcal{C}\text{-CC}_G$ as well, thus there is $L \in \mathcal{N_C}(G)$ such that $C_{\overline{A}_0}(\overline{x}_0) \subseteq \psi\left(C_{A_0}(x_0)K\right)$ in $G/L$. This is equivalent to $\overline{b}C_{\overline{A}_0}(\overline{x}_0) \subseteq \psi\left( b C_{A_0}(x_0)K\right)$.
	
	Base of the induction: let $n = 1$. First suppose that $bC_{A_0}(x_0)\cap x_1 A_1 y_1 = \emptyset$. This is equivalent to $x_1 \not\in b C_{A_0}(x_0)y_1^{-1}A_1$. Since $A_0 = h A h^{-1}$ then $C_{A_0}(x_0) = h C_A(g) h^{-1}$ where $g= h^{-1} x_0 h$. Thus $x_1 \not\in b hC_{A}(g)h^{-1}y_1^{-1}A_1 (y_1 h) (y_1 h)^{-1}$. Set $D = (y_1 h)^{-1} A_1 (y_1 h)$. Thus we have $x_1 \not\in b h C_{A}(g)D h^{-1} y^{-1}_1$.
	
	By theorem \ref{intersection} we see that $A\cap A_1$ and $A \cap D$ are conjugates of full subgroups. Thus for arbitrary $f\in G$ we have that $f^{A\cap D}$ (or $f^{A_1 \cap A}$) and the pair $(A \cap D, f)$ (or $(A_1 \cap A, f)$) has $\mathcal{C}\text{-CC}_G$ for all $f\in G$, thus by Lemma \ref{ashot_5.7} we see that the double coset $C_A(g)D$ is $\mathcal{C}$-separable in G. Then $bh C_A(g)D h^{-1}y^{-1}$ is $\mathcal{C}$-separable as well. Equivalently $bC_{A_0}(x_0)y_1^{-1}A_1$ is $\mathcal{C}$-closed in $G$ and thus there is $N \in \mathcal{N_C}(G)$ such that $x_1 \not\in (bC_{A_0}(x_0)y_1^{-1}A_1)N$. By replacing $K \cap N$ we can assume $N \leq K$. Since the pair $(A_0, x_0)$ has $\mathcal{C}\text{-CC}_G$ there is $L \in \mathcal{N_C}(G)$ such that $L \leq N$ and $C_{\overline{A}_0}(\overline{x}_0) \subseteq \psi\left(C_{A_0}(x_0)N\right) \subseteq \psi\left(C_{A_0}(x_0)K\right)$ in $G/L$ where $\psi: G \rightarrow G/L$ is the natural projection and $\overline{A}_0 = \psi(A_0)$, $\overline{x}_0 = \psi(x_0)$.
	
	This means that $\psi^{-1}(\overline{b})C_{\overline{A}_0}(\overline{x}_0)\overline{y}_1^{-1}\overline{A}_1) \subseteq b C_{A_0}(x_0)y_1^{-1}A_1N$, where $\overline{b} = \psi(b)$, and thus from construction of $N$ we see that $x_1 \not\in \psi^{-1}(\overline{b}C_{\overline{A}_0}(\overline{x}_0)\overline{y}_1^{-1}\overline{A}_1)$ which concludes that $\overline{x}_1 \not\in \overline{b}C_{\overline{A}_0}(\overline{x}_0)\overline{y}_1^{-1}\overline{A}_1$, thus $\overline{b}C_{\overline{A}_0}(\overline{x}_0) \cap \overline{x}_1 \overline{A}_1 \overline{y}_1 = \emptyset$ and therefore
\begin{displaymath}
	\emptyset = \overline{b}C_{\overline{A}_0}(\overline{x}_0) \cap \overline{x}_1 \overline{A}_1 \overline{y}_1 \subseteq \psi\left( (bC_{A_0}(x_0) \cap x_1 A _1 y_1) K \right).
\end{displaymath}

Now suppose $b C_{A_0}(x_0) x_1 A_1 y_1 \neq \emptyset$.	By Remark \ref{pomocna} we see that $b C_{A_0}(x_0) x_1 A_1 y_1 = a(C_{A_0}(x_0) \cap y_1^{-1} A_1 y_1)$, where $a \in b C_{A_0}(x_0) \cap x_1 A_1 y_1$. Clearly $C_{A_0}(x_0) \cap y_1^{-1} A_1 y_1 = C_{E}(x_0)$, where $E = A_0 \cap y^{-1} A_1 y$. By Theorem \ref{intersection} we see that $E$ is a conjugate of some full subgroup of $G$ and thus the pair $(E, x_0)$ has $\mathcal{C}\text{-CC}_G$ and therefore there is $M \in \mathcal{N_C}(G)$ such that $M \leq K$ and
\begin{equation}
	C_{\varphi(E)}(\varphi(x_0)) \subseteq \varphi(C_E(x_0) K) = \varphi\left((C_{A_0}(x_0)\cap y_1^{-1} A_1 y_1)K\right) \mbox{ in } G/M,
\end{equation}
	\label{prvni}
	where $\varphi: G \rightarrow G/M$ is the natural projection. However, we need to have control over $\overline{a}(C_{\overline{A}_0}(\overline{x}_0)\cap \overline{y}_1^{-1}\overline{A}_1\overline{y}_1)$. The full subgroups $A, A_1$ are retract whose corresponding retractions commute and their intersection, $A \cap A_1$, has $\mathcal{C}\text{-CC}_G$ because it is a full subgroup. Set $x = y_1 h$. By Lemma \ref{ashot_5.6} there is $L \in \mathcal{N_C}(G)$ such that $L \leq M$ and
	\begin{displaymath}
		\psi(A)\cap \psi(x A_1 x^{-1}) \subseteq \psi(A \cap x A_1 x^{-1})\psi(M) \mbox{ in }G/L
	\end{displaymath}
where $\psi \colon G \to G/L$ is the natural projection. It can be easily checked that
\begin{displaymath}
	\overline{A}_0 \cap \overline{y}_1^{-1}\overline{A}_1\overline{y}_1 = \overline{h}\overline{A}\overline{h}^{-1}\cap \overline{h} \left(\overline{y}_1 \overline{h} \right)^{-1}\overline{A}_1 \left(\overline{y}_1\overline{h}\right)\overline{h}^{-1}= \overline{h}(\overline{A}\cap \left(\overline{y}_1 \overline{h})^{-1} \overline{A}_1\left(\overline{y}_1 \overline{h}\right)\right)\overline{h}^{-1},
\end{displaymath}
where $\overline{h} = \psi(h)$. Thus
\begin{equation}
\label{druha}
\begin{split}
	\overline{A}_0 \cap \overline{y}_1{-1}\overline{A}_1\overline{y}_1 &\subseteq \psi(h)\left[ \psi\left(A \cap (y_1 h)^{-1}A_1(y_1h)\right)\psi(M)\right]\psi(h)^{-1}\\
	&=\psi(hAh^{-1} \cap y_1^{-1}Ay_1)\psi(M)\\
	&=\psi(E)\psi(M) = \psi(EM).
\end{split}
\end{equation}
	Since $\psi(a) = \overline{a} \in \overline{b}C_{\overline{A}_0}(\overline{x}_0)\cap \overline{x}_1 \overline{A}_1 \overline{y}_1$ we can use Remark \ref{pomocna} and write
	\begin{displaymath}
		\overline{b}C_{\overline{A}_0}(\overline{x}_0)\cap \overline{x}_1 \overline{A}_1 \overline{y}_1 = 
		\overline{a}(C_{\overline{A}_0}(\overline{x}_0) \cap \overline{y}_1^{-1}\overline{A}_1\overline{y}_1). 
	\end{displaymath}
Again, $C_{\overline{A}_0}(\overline{x}_0) \cap \overline{y}_1^{-1}\overline{A}_1\overline{y}_1 = C_{\overline{A}_0 \cap \overline{y}_1^{-1}\overline{A}_1\overline{y}_1}(\overline{x}_{0})$. Since $\overline{A}_0 \cap \overline{y}_1^{-1}\overline{A}_1\overline{y}_1 \subseteq \psi(EM)$ we get that
\begin{equation}
	\label{treti}
	\overline{b}C_{\overline{A}_0}(\overline{x}_0)\cap \overline{x}_1 \overline{A}_1 \overline{y}_1 \subseteq \overline{a}C_{\psi(EM)}(\overline{x}_0).
\end{equation}
Let $\varphi\colon G \to G/M$ be the natural projection. Since $L \leq M$ there is unique homomorphism $\xi \colon G/L \to G/M$ such that $\varphi = \xi \circ \psi$. Clearly $\psi(M) = \ker(\xi)$ and thus $\xi(\psi(EM)) = \xi(\psi(E)\psi(M)) = \xi(\psi(E)) = \varphi(E)$, also $\xi(\overline{x}_0)=\varphi(x_0)$. Therefore for arbitrary $z \in C_{\psi(EM)}(\overline{x}_0)$ in $G/L$ we have
\begin{displaymath}
	\xi(z) \in C_{\varphi(E)}(\varphi(x_0)) \subseteq \varphi(C_E(x_0)K) = \xi(\psi(C_E(x_0)K)).
\end{displaymath} 
	Altogether this means that $z \in \psi(C_E(x_0)K)\ker(\xi) = \psi(C_E(x_0)K)\psi(M) = \psi(C_E(x_0)K)$.
	Thus we get $C_{\psi(EM)}(\overline{x}_0)\subseteq \psi(C_E(x_0)K)$. Combined with (\ref{treti}) we get
\begin{displaymath}
	\overline{b}C_{\overline{A}_0}(\overline{x}_0) \cap \overline{x}_1 \overline{A}_1 \overline{y}_1 \subseteq \overline{a}\psi(C_E(x_0)K) = \psi(a C_E(x_0)K) = \psi((bC_{A_0}(x_0)\cap x_1 A_1 y_1)K).
\end{displaymath}

Now suppose $n>1$ and that the result has been proved for all $m \leq n-1$.

If $b C_{A_0}(x_0) \cap \bigcap_{i=1}^{n-1} x_i A_i y_i = \emptyset$ then by induction there is $L \in \mathcal{N_C}(G)$ such that $L \leq K$ and
\begin{displaymath}
	\overline{b} C_{\overline{A}_0}(\overline{x}_0) \cap \bigcap_{i=1}^{n-1} \overline{x}_i {A}_i \overline{y}_i \subseteq \psi\left(\left(b C_{A_0}(x_0) \cap \bigcap_{i=1}^{n-1} x_i A_i y_i\right) K\right) = \emptyset \mbox{ in } G/L.
\end{displaymath}
Clearly $\overline{b} C_{\overline{A}_0}(\overline{x}_0) \cap \bigcap_{i=1}^{n} \overline{x}_i {A}_i {y}_i \subseteq \overline{b} C_{\overline{A}_0}(\overline{x}_0) \cap \bigcap_{i=1}^{n-1} \overline{x}_i {A}_i {y}_i = \emptyset$ and thus we are done.

Now suppose that	$b C_{A_0}(x_0) \cap \bigcap_{i=1}^{n-1} x_i A_i y_i \neq \emptyset$ in $G$. Then for any $a \in b C_{A_0}(x_0) \cap \bigcap_{i=1}^{n-1} x_i A_i y_i$ we can use Remark \ref{pomocna} to see that $b C_{A_0}(x_0) \cap \bigcap_{i=1}^{n-1} x_i A_i y_i = a C_E(x_0)$ where $E = A_0 \cap \bigcap_{i=1}^{n-1}y_i^{-1} A_i y_i$.

Then $b C_{A_0}(x_0) \cap \bigcap_{i=1}^{n} x_i A_i y_i = aC_E(x_0) \cap x_n A_n y_n$ thus by using the base case of the induction we can get that there is $M \in \mathcal{N_C}$ such that $M \leq K$ and
\begin{equation}
	\label{ctvrta}
	\phi(a)C_{\phi(E)}(\phi(x_0)) \cap \phi(x_n A_n y_n) \subseteq \phi((a C_E(x_0)\cap x_n A_n y_n)K) \mbox{ in } G/M
\end{equation}
where $\phi\colon G \to G/M$ is the natural projection. Also by induction hypothesis there is $L \in \mathcal{N_C}(G)$ such that $L \leq M \leq K \leq G$ and
\begin{displaymath}
	\overline{b}C_{\overline{A}_0}(\overline{x_0}) \cap \bigcap_{i=1}^{n-1}\overline{x}_i \overline{A}_i \overline{y}_i \subseteq \psi \left( \left(b C_{A_0}(x_0) \cap \bigcap_{i=1}^{n-1}x_i A_i y_i \right)M \right) \mbox{ in } G/L.
\end{displaymath}	
Note that $\ker(\psi) = L \leq M = \ker(\phi)$. Therefore

\begin{align*}
	\psi^{-1}\left(\overline{b}C_{\overline{A}_0}(\overline{x}_0)\cap \bigcap_{i=1}^{n}\overline{x}_i\overline{A}_i\overline{y}_i\right)
				&= \psi^{-1}\left(\overline{b}C_{\overline{A}_0}(\overline{x}_0)\cap \bigcap_{i=1}^{n-1}\overline{x}_i\overline{A}_i\overline{y}_i \right) \cap \psi^{-1}\left(\overline{x}_n\overline{A}_n\overline{y}_n\right)\\
				& \subseteq	\left( b C_{A_0}(x_0) \cap \bigcap_{i=1}^{n-1}x_i A_i y_i\right)M \cap (x_n A_n y_n)L\\
				& \subseteq aC_E(x_0) M \cap x_n A_n y_n M\\
				& \subseteq \phi^{-1}\left[\phi(a)C_{\phi(E)}(\phi(x_0))\right] \cap \phi^{-1}\left[\phi(x_n A_n y_n)\right]\\
				& = \phi^{-1} \left[ \phi(a)C_{\phi(E)}(\phi(x_0))\cap \phi(x_n A_n x_n)\right]
\end{align*}
Using (\ref{ctvrta}) we get
\begin{align*}
			\phi^{-1} \left[ \phi(a)C_{\phi(E)}(\phi(x_0))\cap \phi(x_n A_n x_n)\right]	\subseteq \left(a C_E(x_0)\cap x_n A_n y_n\right)K.
\end{align*}
Finally, this leads us to
\begin{displaymath}
	 \psi^{-1}\left(\overline{b}C_{\overline{A}_0}(\overline{x}_0)\cap \bigcap_{i=1}^{n}\overline{x}_i\overline{A}_i\overline{y}_i\right) \subseteq \left(b C_{A_0}(x_0) \cap \bigcap_{i=1}^{n}x_i A_i y_i\right)K.
\end{displaymath}
Which concludes the proof of the lemma.	
\end{proof}

From now on we assume that Lemmas \ref{hlavni1} and \ref{hlavni2} are hold for all graph products $\Gamma \mathcal{G}$ such that $|V\Gamma| \leq r$. Let $G = \Gamma\mathcal{G}$ be a graph product such that $|V\Gamma| = r+1$. Let $v\in V\Gamma$ be arbitrary and set $A = G_{V\Gamma \setminus \{v\}}$, $H = G_{\link(v)}$ and $C = G_v$. Then by Remark \ref{structure_of_graph_products} we see that $G = A \specam_{H} C$. Also, suppose that $B \leq A$ is a full subgroup of $A$.

\begin{lemma}
	\label{8.7}
	Let $g \in G \setminus A$ and $f \in G \setminus g^B$. Then there are homomorphisms $\psi_A \colon A \to Q$ and $\psi_C \colon C \to S$ where $Q, S \in \mathcal{C}$ such that for the corresponding extension $\psi \colon G \to P$, where $P = Q \specam_{\psi_A(H)} S$, we have $\psi(f)\not\in \psi(g)^{\psi(B)}$.
\end{lemma}
\begin{proof}
	Let $g = x_0 c_1 x_1 \dots c_n x_n$, where $x_0, x_1, \dots, x_n \in A$ and $c_1, \dots, c_n \in C$, and $f = y_0 d_1 y_1 \dots y_m d_m$, where $y_0, y_1, \dots, y_m \in A$ and $d_1, \dots, d_m \in C$, be the reduced expressions for $g$ and $f$ respectively. Since $g \not \in A$ we see that $n \geq 1$. We have to consider four separate cases.
	
	Case 1: suppose $n \neq m$. Since $G$ is residually-$\mathcal{C}$ by Lemma \ref{GP_residually_P} (and $A$ as well) and $H$ is a full subgroup of $A$, $H$ is a retract in $A$ and thus is $\mathcal{C}$-closed in $A$ by Lemma \ref{retract_restriction}. Thus there is $L \in \mathcal{N_C}(A)$ such that $\psi_A(x_i) \not \in \psi_A(H)$ whenever $x_i \not\in H$ and $\psi_A(y_j) \not\in \psi_A(H)$ whenever $y_j \not\in H$, where $\psi_A \colon A \to A/L$ is the natural projection. Since $C$ is a vertex group we have that $C$ is residually-$\mathcal{C}$ by assumption and thus there is $M \in \mathcal{N_C}(C)$ such that $\psi_C(c_i) \neq 1$ and $\psi_C(d_j) \neq 1$ for $i = 1,2,\dots, n$, $j = 1, 2, \dots, m$ where $\psi_C \colon C \to C/M$ is the natural projection. Then for the corresponding extension $\psi: A\specam_H C \to A/L \specam_{\psi_A(H)} C/M$ we have
\begin{equation*}
	\begin{split}
		\psi(g) &= \psi_A(x_0)\psi_C(c_1)\psi_A(x_1) \dots \psi_C(c_n) \psi_A(x_n),\\
		\psi(f) &= \psi_A(y_0)\psi_C(d_1)\psi_A(y_1) \dots \psi_C(d_m) \psi_A(y_m).
	\end{split}
\end{equation*}
These are again reduced expressions and $n \neq m$. Then by Lemma \ref{conjugacy_criterion_special_amalgams} we see that $\psi(g) \not\in \psi(f)^{\psi(B)}$.

Case 2: $n = m$ and $c_j \neq d_j$ for some $j$. Again by argumentation analogous to previous case we see that there are $L \in \mathcal{N_C}(A)$ and $M \in \mathcal{N_C}(C)$ such that $\psi_A(x_i) \not\in \psi_A(H)$ whenever $x_i \not\in H$, $\psi_A(y_i) \not\in \psi_A(H)$ whenever $y_i \not\in H$ and $\psi_C(c_j) \neq \psi_C(d_j) $ where $\psi_A \colon A \to A/L$ $\psi_C \colon C \to C/M$ are the corresponding natural projections. Then for the corresponding extension $\psi: A\specam_H  C \to A/L \specam_{\psi_A(H)} C/M$ we have
\begin{equation*}
	\begin{split}
		\psi(g) &= \psi_A(x_0)\psi_C(c_1)\psi_A(x_1) \dots \psi_C(c_n) \psi_A(x_n),\\
		\psi(f) &= \psi_A(y_0)\psi_C(d_1)\psi_A(y_1) \dots \psi_C(d_n) \psi_A(y_n).
	\end{split}
\end{equation*}
These are again reduced expressions and $\psi(c_j) \neq \psi(d_j)$. Then again by Lemma \ref{conjugacy_criterion_special_amalgams} we see that $\psi(g) \not\in \psi(f)^{\psi(B)}$.

Case 3: $n=m$, $c_i = d_i$ for $i = 1, 2, \dots, n$ and $x_0 x_1 \dots x_n \not\in (y_0 y_1 \dots y_n)^B$. Since $x = x_0 x_1 \dots x_n \in A$ and $B$ is a full subgroup of $A$ we see that $x^B$ is $\mathcal{C}$-closed in $G$ by Lemma \ref{hlavni2} and therefore there is $L \in \mathcal{N_C}(A)$ such that $\psi_A(x) \not \in \psi_A(g)^{\psi_A(B)}$, where $\psi_A \colon A/L$ is the natural projection. Since $C$ is a vertex group we know by assumption that it is residually-$\mathcal{C}$ and thus there is $M \in \mathcal{N_C}(C)$ such that $\psi_C(c_i) \neq 1 \neq \psi_C(d_i)$ for $i = 1, 2, \dots, n$, where $\psi_C \colon C \to C/M$ is the natural projection. By extending $\psi_A \colon A \to Q = A/L$  and $\psi_C \colon C \to S = C/M$ to $\psi \colon A \specam_H C \to Q \specam_{\psi_A(H)} S$ we get that $\psi(f) \not \in \psi(g)^{\psi(B)}$ by Lemma \ref{conjugacy_criterion_special_amalgams}.

Case 4: Now we assume that condition (i) and (ii) from conjugacy criterion for special amalgams are satisfied and (iii) is not. Namely: let $b_0 \in B$ be such that $bxb^{-1} = y$ and $I  = \emptyset$ where
\begin{equation*}
 I = b_0C_B(x) \cap x_0Hx_0^{-1} \cap (y_0 y_1)H(x_0x_1)^{-1} \cap \dots \cap (y_0 y_1 \dots y_{n-1})H(x_0 x_1 \dots x_{n-1})^{-1},
\end{equation*}
	where $x = x_0 x_1 \dots x_n$. By assumption $H$ is a full subgroup of $A$ and thus it is $\mathcal{C}$-closed in A by Lemma \ref{retract_restriction}, hence there is $K \in \mathcal{N_C}(A)$ such that $x_iK \cap H = \emptyset = y_i K \cap H$ for all $i = 0, 1, \dots, n$. We assume that Lemmas \ref{hlavni1} and \ref{hlavni2} are true for $A$ and thus assumptions of Lemma \ref{humus} are true for $A$. Therefore we can use Lemma \ref{humus} to see that there is $L \in \mathcal{N_C}(A)$ such that $L \leq K$ and
\begin{equation*}
	\overline{b}_0C_{\overline{B}}(\overline{x}) \cap \overline{y}_0 H \overline{x}^{-1}_0 \cap (\overline{x}_0 \overline{x_1})\overline{H}(\overline{y}_0\overline{y}_1)^{-1} \cap \dots \cap (\overline{x}_0 \overline{x}_1 \dots \overline{x}_{n-1})\overline{H}(\overline{y}_0 \overline{y}_1 \dots \overline{y}_{n-1})^{-1} \subseteq \psi_A(IK),
\end{equation*}
where $\psi_A \colon A \to A/L$ is the natural projection and $\overline{x} = \psi_A(x)$, $\overline{x}_i = \psi_A(x_i)$, $\overline{y}_i = \psi_A(y_i)$ for $i = 0, 1, \dots n-1$, $\overline{H} = \psi_A(H)$ and $\overline{B} = \psi_A(B)$. Note that since $I = \emptyset$ we have $\psi_A(IK) = \emptyset$. Also since $C$ is a vertex group we know it is $\mathcal{C}\text{-HCS}$ and thus residually-$\mathcal{C}$, hence there is $M \leq \mathcal{N_C}(C)$ such that $\psi_C(x_i) \neq 1 \neq \psi_C(y_i) \neq 1$ for $i = 1,2,\dots, n$, where $\psi_C \colon C \to C/M$ is the natural projection. Therefore if we extend $\psi_A \colon A \to Q = A/L$ and $\psi_C \colon C \to S = C/M$ to $\psi \colon A \specam_H C \to Q \specam_{\psi_A(H)} S$ we see that $\psi(f) \not\in \psi(g)^{\psi(B)}$ by Lemma \ref{conjugacy_criterion_special_amalgams} as the condition (iii) is not true for $\psi(f)$ and $\psi(g)$.
\end{proof}

\begin{lemma}
\label{8.8}
	Suppose that $g_0, f_0, f_1, \dots , f_n \in G$ are arbitrary and that the products $g_0 = c_1 x_1 \dots c_m x_m$, where $c_1, \dots, c_m \in C$ and $x_1, \dots, x_m \in A$, and $f_0 = d_1 y_1 \dots d_k y_k$, where $d_1, \dots, d_k \in C$ and $y_1, \dots, y_k \in A$, are cyclically reduced in $G$. If $f_j \not\in g^H$ for all $j = 1, 2, \dots, m$ then there are groups $Q, S \in \mathcal{C}$ and a epimorphisms $\psi_A \colon A \to Q$ and $\psi_C \colon C \to S$ such that for the extension $\psi \colon A \specam_{H} C \to P$, where $P = Q \specam_{\psi_A(H)} S$, all of the following are true:
	\begin{enumerate}
		\item[(i)] $\psi(f_i) \not\in \psi(g_0)^{\psi(H)}$ in $P$ for all $i = 1, 2, \dots ,n$,
		\item[(ii)] the products 
			\begin{equation*}
				\begin{split}
				\psi(g_0) = \psi_C(c_1) \psi_A(x_1) \dots \psi_C(c_m) \psi_A(x_m),\\
				\psi(f_0) = \psi_C(d_1) \psi_A(y_1) \dots \psi_C(d_k) \psi_A(y_k)
				\end{split}
			\end{equation*}
	 are cyclically reduced in $P$.
	\end{enumerate}
\end{lemma}
\begin{proof}
	We set $B := H$ which is a full subgroup of $A$ thus we can apply Lemma \ref{8.7} to pairs $(g_0, f_1), \dots, (g_0, f_n)$ to obtain $L_1, \dots, L_n \in \mathcal{N_C}(A)$ and $M_1, \dots M_n \in \mathcal{N_C}(C)$ with corresponding natural projections $\alpha_i \colon A \to A/L_i$ and $\gamma \colon C \to C/M_i$, such that $\psi_i(f_i)\not\in \psi_i(g_0)^{\psi_i(H)}$ in $P_i$, where $\psi_i \colon A \specam_H C \to A/L_i \specam_{\alpha_i(H)} C/M_i$, for $i = 1, 2, \dots, n$.
	Note that since $H$ is a retract of $A$ and $A$ is residually-$\mathcal{C}$ we get that $H$ is $\mathcal{C}$-closed in $A$ by Lemma \ref{retract_restriction}. Thus there is $K \in \mathcal{N_C}(A)$ such that $x_i K \cap H = \emptyset$ whenever $x_i \not \in H$ and $y_j K \cap H = \emptyset$ whenever $y_j \not \in H$. Also by the same argumentation there is $M' \in \mathcal{N}_{\mathcal{C}}(C)$ such that $c_i \not\in M'$ and $d_j \not\in M'$ for all $i = 1, 2, \dots , m$ and $j = 1,2, \dots, k$.
	
	Set $L = K \cap \bigcap_{i=1}^{n}L_i$ and $M = M'\cap_{i=1}^n M_i$. Let $\psi_A \colon A \to A/L$ and $\psi_C \colon C \to C/M$ be the natural projections and let $\psi\colon A\specam_H C \to A/L \specam_{\psi_A(H)}C/M$. Clearly, this is the map we are looking for.
\end{proof}

\begin{lemma}
\label{8.9}
	Let $K \in \mathcal{N_C}(G)$, $B \leq A$ be a full subgroup of $A$ (and thus of $G$). Let $g \in G \setminus A$ be an element with reduced form $g = x_0 c_1 x_1 \dots c_n x_n$, where $x_0, x_1, \dots, x_n \in A$ and $c_1, \dots, c_n \in C$, such that $n \geq 1$. Then there are  groups $Q, S \in \C$ and epimorphisms $\psi_A \colon A \to Q$, $\psi_C \colon C \to S$ with the corresponding extension $\psi \colon A \specam_H C \to P$, where $P = Q \specam_{\psi_A(H)} S$, such that the following are true:
\begin{enumerate}
	\item[(i)] $C_{\psi(B)}(\psi(g)) \subseteq \psi(C_B(g)K)$,
	\item[(ii)] $\ker(\psi_A) \leq A \cap K$, $\ker(\psi_C) \leq C \cap K$ and $\ker(\psi) \leq K$.
\end{enumerate}
\end{lemma}
\begin{proof}
	Since $A$ is residually-$\mathcal{C}$ and $H$ is a retract in $A$ we see that $H$ is $\mathcal{C}$-closed in $A$ by Lemma \ref{retract_restriction}. Thus there is $M_1 \in \mathcal{N_C}(A)$ such that $x_i M_1 \cap H = \emptyset$ for all $i = 1, 2, \dots, n-1$. We may replace $M_1$ by $M_1 \cap (A \cap K)$ to ensure that $M_1 \leq A \cap K$. By Lemma \ref{special_amalgam_centraliser} we have $C_B(g) = I$ where
	\begin{equation*}
		I = C_B(x)\cap x_0H x_0^{-1} \cap (x_0 x_1)H(x_0 x_1)^{-1} \cap \dots \cap (x_0 \dots x_{n-1})H(x_0 \dots x_{n-1})^{-1}
	\end{equation*}
	and $x = x_0 x_1 \dots x_n$. Since $A$ is a graph product with less than $n$ vertices we may assume that both Lemmas \ref{hlavni1} and \ref{hlavni2} hold for $A$. Thus we can use Lemma \ref{humus} to show the there is $L_1 \in \mathcal{N_C}(A)$ such that $L_1 \leq M_1 \leq K \cap A$ and
	\begin{equation*}
		J = C_{\overline{B}}(\overline{x})\cap \overline{x}_0\overline{H} \overline{x}_0^{-1} \cap \dots \cap (\overline{x}_0 \dots \overline{x}_{n-1})\overline{H}(\overline{x}_0 \dots \overline{x}_{n-1})^{-1} \subseteq \psi_A(IM_1) \mbox{ in } A/L_1,
	\end{equation*}
	 where $\psi_A \colon A \to Q = A/L_1$ is the natural projection and $\overline{x} = \psi_A(x)$, $\overline{x}_i = \psi_A(x_i)$, for $i = 0, 1, \dots n-1$, $\overline{H} = \psi_A(H)$ and $\overline{B} = \psi_H(B)$.
	 
	 Since $C$ is $\mathcal{C}\text{-HCS}$ it is also residually-$\mathcal{C}$ and thus there is $Z \in \mathcal{N_C}(C)$ such that $Z \leq  C \cap K$ and $\psi_C(c_i) \neq 1$ in $C/Z$ for $i = 1, 2, \dots, n$, where $\psi_C \colon C \to S = C/Z$ is the natural projection. 
	 
	  Let $P = Q\specam_{\psi_A(H)} S$ and let $\psi\colon A\specam_{H}C \to P$ be the canonical extension of $\psi_A$ and $\psi_C$ to $G$. Now
	 \begin{displaymath}
	 	\psi(g)=\psi_A(x_0) \psi_C(c_1) \psi_A(x_1) \dots \psi_C(c_n) \psi(x_n),
	 \end{displaymath}
thus	 $C_{\overline{B}}(g) = J$ in $P$ by Lemma \ref{special_amalgam_centraliser}. This means that 
\begin{displaymath}
	C_{\psi(B)}(\psi(g)) \subseteq \psi(IM_1) = \psi(C_B(g)M_1),
\end{displaymath}
therefore the first assertion of the lemma holds. Note that $\ker(\psi_A) = L_1 \leq M_1 \leq K \cap A$ and $\ker(\psi_C) \leq K \cap C$.
 Since $\ker(\psi) = \lnorm \ker(\psi_A) \ker(\psi_C) \rnorm^{G}$ by Lemma \ref{extension_kernel} we see that $\ker(\psi) \leq K$ and thus the second assertion holds as well.
\end{proof}

\begin{lemma}
\label{8.10}
	Let $K \in \mathcal{N_C}(G)$ and let $g = c_1 x_1 \dots c_n x_n$, where $c_1, \dots c_n \in C$ and $x_1, \dots, x_n \in A$, be a cyclically reduced element of $G$ with $n \geq 1$. Then there are homomorphisms $\psi_A \colon A \to Q$ and $\psi_C \colon C \to S$, where $Q, S \in \mathcal{C}$, with a corresponding extension $\psi \colon G \to P$, where $P = Q \specam_{\psi_A(H)} S$, such that the following is true 
\begin{enumerate}
	\item[(i)] $\ker(\psi_A) \leq A \cap K$, $\ker(\psi_C) \leq C \cap K$ and $\ker(\psi) \leq K$,
	\item[(ii)]	$C_{P}(\psi(g)) \subseteq \psi(C_G(g)K) \mbox{ in } P$.
\end{enumerate}
\end{lemma}
\begin{proof}
	We need to consider two separate cases: $x_n \in H$ or $x_n \not \in H$.
	
	Suppose $x_n \in H$.
	Then by Lemma \ref{7.8} we see that $n=1$ and $C_G(c_1 x_1) = C_C(c_1) \times C_H(x_1)$. Since $A$ is a graph product with less than $n$ vertices we can use the induction hypothesis of Lemma \ref{hlavni2} to find $L \in \mathcal{N_C}(A)$ such that $L \leq K \cap A$ and $C_{\alpha(H)}(\alpha(x_1)) \subseteq \alpha(C_H(x_1)( K \cap A))$, where $\alpha \colon A \to Q = A/L$ is the natural projection. Since $C$ is a vertex group we assume that it is $\mathcal{C}\text{-HCS}$ and therefore it satisfies $\mathcal{C}\text{-CC}$ by Theorem \ref{cc_and_cs_is_hcs}. This means that there is $M \in \mathcal{N_C}(C)$ such that $M \leq K \cap C$ and $C_S(\gamma(c_1)) \subseteq \gamma(C_C(c_1)(K\cap C))$, where $S = C/M$ and $\gamma \colon C \to S$ is the canonical projection. Note that since $A$ is residually-$\mathcal{C}$ by Lemma \ref{GP_residually_P} and $C$ is residually-$\mathcal{C}$ by assumption these maps can be chosen so that $\alpha(x_1) \neq 1$ in $Q$ and $\gamma(c_1)\neq 1$ in $S$.
	
	 Let $\psi \colon A\specam_H C \to P$, where $P = Q\specam_{\alpha(H)}S$, be the canonical extension of $\alpha$ and $\gamma$ to $G$. Since $\psi(g) = \gamma(c_1) \alpha(x_1)$ is again reduced by Lemma \ref{7.8} we see that $C_P(g) = C_S(\gamma(c_1)) \times C_P(\alpha)$. From the construction of the maps $\alpha$ and $\gamma$ we see that $C_{P}(\psi(g)) \subseteq \psi(C_G(g)K)$. 

	Suppose $x_n \not\in H$.
	Let $\{p_1, p_2, \dots\, p_{n+1}\}$ be the set of prefixes of $g$ and assume there is $1 \leq m < n$ such that $p_i^{-1} g p_i \not\in g^H$ if $i \leq m$ and $p_i^{-1} g p_i \in g^H$ if $i > m$. For every $i > m$ there is $h_i \in H$ such that $h_i p_i^{-1} g p_i h^{-1} = g$ and thus $h_i p_i^{-1} \in C_G(g)$. Set 
	\begin{equation}
		\label{omega}
		\Omega = \{h_i p_i^{-1} \mid i >m\}.
	\end{equation}
	Let $f_i = p_i^{-1} g p_i$ for $i = 1, 2, \dots, m$. Clearly $f_1, f_2, \dots, f_m \not\in g^H$. Let $f_0 = g$. Using Lemma \ref{8.8} we see that there are $\mathcal{C}$-groups $Q_1$ and $S_1$ and epimorphism $\alpha_1 \colon A \to Q_1$ and $\gamma_1 \colon C \to S_1$ such that if we take the corresponding extension $\psi_1 \colon A \specam_H C \to Q_1 \specam_{\alpha_1(H)} S_1$ the element $\psi_1(f_0) = \gamma(c_1) \alpha(x_1) \dots \gamma(c_n) \alpha(x_n)$ is cyclically reduced and $\psi_i(f_i) \not\in \psi(g)^{\psi(H)}$ for $i=1, 2, \dots, m$. Let $L_1 = \ker(\alpha_1)$ and $M_1 = \ker(\gamma_1)$. Note that $L_1 \in \mathcal{N_C}(A)$ and $M_1 \in \mathcal{N_C}(C)$.
	
	By Lemma \ref{8.9} there are $\mathcal{C}$-groups $Q_2, S_2$ and an epimorphism $\alpha_2 \colon A \to Q_2$, $\gamma_2 \colon C \to S_2$ such that if we take the corresponding extension $\psi_2 \colon A\specam_H C \to Q_2 \specam_{\alpha_2(H)} S_2$ we have $\ker(\alpha_2) \leq K \cap A$, $\ker(\gamma_2) \leq K \cap C$, $\ker(\psi_2) \leq K$ and
	\begin{displaymath}
		C_{\psi_2(H)}(\psi_2(g)) \subseteq \psi_2(C_H(g)K) \mbox{ in } Q_2 \specam_{\alpha_2(H)} S_2.
	\end{displaymath}
	
	Take $L = \ker(\alpha_1) \cap \ker(\alpha_2) \normleq A$ and let $M = \ker(\gamma_1) \cap \ker(\gamma_2) \normleq C$. Let $\alpha \colon A \to Q = A/L$ and $\gamma \colon C \to S = C/M$  be the natural projections. Note that $L \in \mathcal{N_C}(A)$ and $M \in \mathcal{N_C}(C)$. Let $\psi \colon G \to P$ be the corresponding extension, where $P = Q \specam_{\alpha(H)} S$. We immediately get
	\begin{enumerate}
		\item $\ker(\alpha) \leq K \cap A$ and $\ker(\psi) \leq K$,
		\item $C_{\psi(H)}(\psi(g)) \subseteq \psi(C_H(g)K)$ in $P$,
		\item $\psi(f_1), \dots, \psi(f_m) \not\in \psi(g)^{\psi(H)}$,
		\item the element $\psi(g) = \gamma(c_1) \alpha(x_1) \dots \gamma(c_n) \alpha(x_n)$ is cyclically reduced in $P$. 
	\end{enumerate}
	Since we assume that $x_n \not\in H$ we get by Lemma \ref{7.8} that
	\begin{equation*}
		C_G(g) = C_H(g)\langle g \rangle \Omega,
	\end{equation*}
	where $\Omega$ is given by (\ref{omega}).
	Also by Lemma \ref  {7.8} we see that in $P$ we have
	\begin{equation*}
		C_P(\psi(g)) = C_{\psi(H}(\psi(g))\langle \psi(g) \rangle \tilde{\Omega}
	\end{equation*}
	where $\tilde{\Omega} = \{\overline{h_i} \overline{p_i}^{-1} \mid i>m \}$  such that $\overline{p}_i$ is a prefix of $\psi(g)$ such that $\overline{p}_i^{-1} \psi(g) \overline{p}_i \in \psi(g)^{\psi(H)}$ and $\overline{h_i} \in \psi(H)$ such that $\overline{h}_i \overline{p}_i \psi(g) \overline{p_i} \overline{h}_i^{-1} = \psi(H)$. Since $\psi(f_i) = \psi(p_i)^{-1}\psi(g)\psi(p_i) \not\in \psi(g)^{\psi(H)}$ when $i \leq m$, therefore we can conclude that $\psi(p_i)^{-1}\psi(g)\psi(p_i) \in \psi(g)^{\psi(H)}$ if and only if $p_i^{-1} g p_i \in g^H$. Altogether we see that $\tilde{\Omega} = \psi(\Omega)$.
	
	Finally, $\langle \psi(g) \rangle = \psi(\langle g \rangle)$. From this we see that
	\begin{displaymath}
		C_P(\psi(g)) \subseteq \tilde{\psi}(C_G(g)K)
	\end{displaymath}  
	and thus the lemma holds.
\end{proof}

\begin{proof}[Proof of Lemma \ref{hlavni1}]
We will proceed by induction on $|V\Gamma|$. If $|V\Gamma| = 0$ then $G = \{1\}$ and the statement holds trivially. Now suppose that the statement holds for all graph products $\Gamma\mathcal{G}$ with $|V \Gamma| \leq r-1$. Let $G = \Gamma\mathcal{G}$ where $|V\Gamma| = r$. There are two cases to be distinguished: $B \neq G$ and $B = G$.

Suppose $B$ is a proper full subgroup of $G$. Then we can pick a maximal proper full subgroup $A \leq G$ such that $B \leq A$. If $g \in A$ then $g^B$ is $\mathcal{C}$-closed in $A$ by induction hypothesis and thus it is $\mathcal{C}$-closed in $G$ by Lemma \ref{retract_restriction} as $A$ is a retract in $G$ and $G$ is residually-$\mathcal{C}$ by Lemma \ref{GP_residually_P}. Suppose that $g \in G \setminus A$ and let $f \in G \setminus g^B$ be arbitrary. By Lemma \ref{8.7} there are $\mathcal{C}$-groups $Q$, $S$ and epimorphism $\alpha \colon A \to Q$, $\gamma \colon C \to S$ with the corresponding extension $\psi \colon G \to Q \specam_{\alpha(H)} S$ such that $\psi(f) \not\in \psi(g)^{\psi(B)}$ in $P$. Since $P$ is a special amalgam of (finite) $\mathcal{C}$-groups we see that it is residually-$\mathcal{C}$ by Lemma \ref{my_dyer}. Since $|Q| < \infty$ we see that $\psi(B)$ is a finite subset of $Q$ and therefore $\psi(g)^{\psi(B)}$ is finite and thus is $\mathcal{C}$-closed in $P$. By Lemma \ref{closed_simplification} we see that $g^B$ is $\mathcal{C}$-closed in $G$.
  
 Now suppose $B=G$. If $g=1$ then $1^G = \{1\}$ is $\mathcal{C}$-separable in $G$ since it is finite subset of $G$ and $G$ is residually-$\C$. Let's assume $g \neq 1$. Then by Lemma \ref{maximal_full_subgroup} there is a maximal full subgroup $A \leq G$ such that $g \not \in A^G$. Then $G$ naturally splits as $G = A \specam_{H} C$ where $H$ is a full subgroup of $A$ and $C$ is a vertex group. Then $g$ is a conjugate to some cyclically reduced element of $G$, say $g_0$. Suppose $g_0 = c_1 x_1 \dots c_n x_n$, where $x_1, \dots, x_n \in A$ and $c_1, \dots, c_n \in C$, is the cyclically reduced expression for $g_0$ Note that $g^G = g_0^G$. Let $f \in G \setminus g^G$. There are two sub-cases to consider: $f \not\in A^G$ and $f \in A^G$.

Suppose $f\not\in A^G$. Let $f_0$ be a cyclically reduced element of $G$ conjugate to $f$, thus $f^G = f_0^G$. Let $f_0 = d_1 y_1 \dots d_m y_m$, where $y_1, \dots, y_m \in A$ and $d_1, \dots, d_m \in C$, be the reduced expression for $f_0$ and let $f_1, f_2, \dots f_m$ denote the set of all of its cyclic permutations. Clearly $f_i \not \in g^H$ for all $i$ since $f \not\in g^G$. Then by Lemma \ref{8.8} there are groups $Q, S \in \C$ and epimorphisms $\alpha \colon A \to Q$, $\gamma \colon C \to S$ with corresponding extension $\psi \colon G \to P$, where $P = Q\specam_{\alpha(H)} S$, such that $\psi(f_1), \psi(f_2), \dots \psi(f_m) \not\in \psi(g)^{\psi(H)}$ and $\psi(f_0) = \gamma(d_1)\alpha(x_1) \dots \gamma(d_m)\alpha(x_m)$ is cyclically reduced in $P$. Since $\psi(f_1), \dots, \psi(f_m)$ are all the cyclic permutations of $\psi(f_0)$ we can conclude that $\psi(f) \not \in \psi(g)^{P}$ by Lemma \ref{collinss_lemma}. 

Assume that $f \in A^G$. By Lemma \ref{8.8} there are groups $Q,S \in \mathcal{C}$ and projections $\alpha \colon A \to Q$, $\gamma \colon C \to S$ with extension $\psi \colon G \to P$, where $P = Q \specam_{\alpha(H)} S$, such that $\psi(g_0) = \gamma(c_1)\alpha(x_1) \dots \gamma(c_n)\alpha(x_n)$ is cyclically reduced in $P$. Since $n \geq 1$ by Lemma \ref{collinss_lemma} we see that $\psi(g_0) \not\in Q^P = \psi(A^G)$. As we assume that $f \in A^G$ we see that $\psi(f^G) = \psi(f)^P \subseteq \psi(A^G)$. We see that $\psi(g_0) \not \in \psi(f)^P$ and hence $\psi(f) \not\sim_P \psi(g)$. 

Either way, in both cases when $f \not\in A^G$ and $f \in A^G$ we have found a homomorphism $\psi$ that separates $f$ from $g^G$ in an amalgam of $\mathcal{C}$-groups which is $\mathcal{C}\text{-HSC}$ by Lemma \ref{my_dyer}. Thus by Lemma \ref{closed_simplification} we see that $g^G$ is $\mathcal{C}$-closed in $G$.
\end{proof}

\begin{proof}[Proof of Lemma \ref{hlavni2}]
Again, we proceed by induction on $|V\Gamma|$. If $|V\Gamma| = 0$ then $G = \{1\}$ and the statement holds trivially. Now suppose that the statement holds for all $\Gamma$ with $|V \Gamma| \leq r-1$. Let $G = \Gamma\mathcal{G}$ where $|V\Gamma| = r$. Let $K \in \mathcal{N_C}(G)$ be arbitrary. There are two cases to be distinguished: $B \neq G$ and $B = G$.

Suppose $B$ is a proper full subgroup of $G$. Then there is a maximal full subgroup $A \leq G$ such that $B \leq A$. Clearly $A$ is a graph product with $r-1$ vertices and therefore the statement holds for $A$. Let $K_A = K \cap A$ and let $K_C = K \cap C$. Obviously $G$ splits as $G = A \specam_H C$, where $H$ is a full subgroup of $A$ and $C$ is a vertex group. We consider two separate sub-cases: $g \in A$ and $g \in G \setminus A$.

Assume that $g \in A$. By induction we see that the pair $(B,g)$ has $\mathcal{C}\text{-CC}_A$ in $A$. Thus there is a $\mathcal{C}$-group $Q$ such that $L_1 = \ker(\alpha) \leq K_A$, and
\begin{displaymath}
	C_{\alpha(B)}\subseteq \alpha(C_B(g)K_A) \mbox{ in } Q,
\end{displaymath}
 where $\alpha \colon A \to Q$ is the natural projection. Let $\rho_A \colon G \to A$ be the canonical retraction of $G$ onto $A$ and set $L = \rho_A^{-1}(L_1) \cap K$. Clearly $L \in \mathcal{N_C}(G)$ and $\rho_A(L) = L_1 \leq K_A$. Let $\phi \colon G \to R = G/L$ be the natural projection. Note that $\ker(\alpha) = \ker(\phi) \cap A$ in $G$, thus we may assume that $Q \leq R$ and $\phi \restriction A = \alpha$. Then $\alpha(K_A) = \phi(K_A) \subseteq \phi(K)$ in $R$. Since $g \in A$, $B$ is a full subgroup of $A$ and $C_{\alpha(B)}\subseteq \alpha(C_B(g)K_1)$ in $Q$, we get that
\begin{displaymath}
	C_{\phi(B)}(\phi(g)) = C_{\alpha(B)}(\alpha(g)) \subseteq \alpha(C_B(g)K_A) \subseteq \phi(C_B(g)K) \mbox{ in } R.
\end{displaymath}
Thus we see that if $g \in A$ then the pair $(B,g)$ has $\mathcal{C}\text{-CC}_G$.

Now suppose that $g\in G \setminus A$. Let $g = x_0 c_1 x_1 \dots c_n x_n$, where $x_0, \dots, x_n \in A$ and $c_1, \dots, c_n \in C$, be a reduced expression for $g$. By Lemma \ref{8.9} we can find $\mathcal{C}$-groups $Q$, $S$ and epimorphisms $\alpha \colon A \to Q$, $\gamma \colon C \cap S$ with corresponding extension $\psi \colon G \to P = Q\specam_{\alpha(H)} S$ such that $\ker(\alpha) \leq K_A$, $\ker(\gamma) \leq K_C$, $\ker(\psi) \leq K$ and
\begin{displaymath}
	C_{\psi(B)}(\psi(g)) \subseteq \psi(C_B(g)K).
\end{displaymath}
P is a special amalgam of $\mathcal{C}$-groups and thus is residually-$\mathcal{C}$ by Corollary \ref{my_dyer}. Since $Q$ is finite we see that $\psi(K) \cap \psi(B) \leq \psi(B) \leq Q$ is finite, thus $\psi(g)^{\psi(B)\cap\psi(K)}$ is $\C$-closed in $P$. By Lemma \ref{3.7} one obtains $\xi \colon P \to R$, where $R \in \mathcal{C}$ such that $\ker(\xi) \leq \psi(K)$ and
\begin{displaymath}
	C_{\xi(\psi(B))}(\xi(\psi(B)))\subseteq \xi(C_{\psi(B)}(\psi(g))\psi(K)) \mbox{ in } R.
\end{displaymath}
Take $\phi \colon G \to R$ to be defined as $\phi = \xi \circ \psi$. Obviously, $\phi$ is the map we are looking for.

We are left with the last remaining case, when $B = G$. We may assume $g \in G\setminus\{1\}$ as the pair $(G,1)$ has $\mathcal{C}\text{-CC}_G$ trivially. By Lemma \ref{maximal_full_subgroup} there is a maximal full subgroup $A \leq G$ such that $g \not \in A^G$. Then $G$ naturally splits as $G = A \specam_{H} C$, where $H \leq A$ is a full subgroup of $A$ and $C$ is a vertex group in $G$. There is $z\in G$ such that $g_0 = z g z^{-1}$ is cyclically reduced in $G$. Let $g_0 = c_1 x_1 \dots c_n x_n$, where $x_i \in A$ and $c_i \in C$ for $i = 1,\dots, n$, be a reduced expression for $g_0$. Since $g \not\in A^G$ we see that $n \geq 1$. By Lemma \ref{8.10} there are $\mathcal{C}$-groups $Q$, $S$ and epimorphisms $\alpha \colon A \to Q$, $\gamma \colon C \to S$ with a corresponding extension $\psi \colon G \to P$, where $P = Q \specam_{\alpha(H)} S$, such that $\ker(\alpha) \leq K \cap A$, $\ker(\gamma) \leq K \cap C$ and
\begin{displaymath}
	C_P(\psi(g)) \subseteq \psi(C_G(g)K) \mbox{ in } P.
\end{displaymath}
Since $P$ is special amalgam of $\mathcal{C}$ groups we see that it is $\mathcal{C}\text{-HCS}$ by Corollary \ref{my_dyer} and thus the pair $(\psi(G), \psi(g))$ satisfies $\mathcal{C}\text{-CC}_P$ in $P$. Note that in every case the homomoprhism $\psi$ was constructed so that $\ker(\psi) \leq K$ thus by Lemma \ref{cc_simplification} we see get that the pair $(g,G)$ has $\mathcal{C}\text{-CC}_G$ in $G$.
\end{proof}

Now we are ready to prove Theorem \ref{main_hcs}.
\begin{proof}[Proof of Theorem \ref{main_hcs}]
	Let $G = \Gamma \mathcal{G}$ be a graph product such that $|V\Gamma| < \infty$ and $G_v$ is $\mathcal{C}\text{-HCS}$ for all $v \in V\Gamma$. Note that $G$ is a full subgroup of itself and thus by Lemma \ref{hlavni1} we see that the pair $(G,g)$ has $\mathcal{C}\text{-CC}_G$ for every $g \in G$ and thus $G$ satisfies $\mathcal{C}\text{-CC}$. By Lemma \ref{hlavni2} we see that the set $g^G$ is $\mathcal{C}$-closed in $G$ for every $g \in G$, hence $G$ is $\mathcal{C}\text{-CS}$. Finally using Theorem \ref{cc_and_cs_is_hcs} we get that $G$ is $\mathcal{C}\text{-HCS}$.
\end{proof}

Note that every group from the class $\C$ is $\mathcal{C}\text{-HCS}$. Then as an immediate consequence of the Theorem \ref{main_hcs} we get that graph products of groups belonging to an extension closed variety of finite groups $\C$ are $\C\text{-HCS}$.
\begin{cor}
	\label{finite}
	Assume that $\C$ is an extension closed variety of finite groups. Let $\Gamma$ be a finite graph and let $\mathcal{G} = \{G_v \mid v\in V\Gamma \}$ be a family of groups such that $G_v \in \mathcal{C}$ for all $v \in V\Gamma$. Then the group $G = \Gamma\mathcal{G}$ is $\mathcal{C}\text{-HCS}$.
\end{cor}

%% file: main_result_and_applications.tex
\section{Infinite graphs and $\CCS$ groups}
\label{CS}
Again, we will assume that the class $\C$ is an extension closed variety of finite groups.
\subsection{Graph products of $\C$-CS groups}

Before we proceed we mention one important property of graph products: they are functorial. 
\begin{remark}
	\label{functorial_property_of_graph_products}
	Let $\Gamma$ be a graph and let $\mathcal{G} = \{G_v | v \in V\Gamma\}$ and $\mathcal{F} = \{F_v | v \in V\Gamma\}$ be two families of groups indexed by vertices of $V\Gamma$. Assume that for every $v \in V\Gamma$ there is a homomorphism $\phi_v \colon G_v \to F_v$. Then there is a unique homomorphism $\phi \colon G \to F$, where $G = \Gamma\mathcal{G}$ and $F = \Gamma \mathcal{F}$ such that $\phi \restriction_{G_v} = \phi_v$ for all $v \in V\Gamma$
\end{remark}

We will use  Corollary \ref{finite} to show that the class of $\mathcal{C}\text{-CS}$ groups is closed under graph products. The main idea is to construct a suitable map onto a finite graph product of groups belonging to the class $\C$. First we need to show that we can always find such a homomorphism that preserves length and support of a given element.
\begin{lemma}
	\label{GP_pomocna}
	Let $G =\Gamma\mathcal{G}$ be a graph product such that $G_v$ is residually-$\mathcal{C}$ for every $v \in V\Gamma$ and let $g \in G$. Then there is $\mathcal{F} = \{F_v | v\in V\Gamma\}$, a family of $\mathcal{C}$-groups indexed by $V\Gamma$, and a homomorphism $\phi_v \colon G_v \to F_v$ for every $v \in V\Gamma$ such that for the corresponding extension $\phi \colon G \to F$ (given by Remark \ref{functorial_property_of_graph_products}), where $F = \Gamma\mathcal{F}$, all of the following are true:
	\begin{itemize}
		\item[(i)] $|g| = |\phi(g)|$,
		\item[(ii)] $\supp(g) = \supp(\phi(g)),$
		\item[(iii)] If $g$ is $\Gamma$-cyclically reduced in $G$ then $\phi(g)$ is $\Gamma$-cyclically reduced in $F$.
	\end{itemize}
\end{lemma}
\begin{proof}
	Let $(g_1, \dots, g_n)$ be a $\Gamma$-reduced expression for $g$ in $G$. For every $v \in V\Gamma$ let $I_v = \{i \mid g_i \in G_v\} \subseteq \supp(g)$ be the set of indices such that the corresponding syllables belong to $G_v$. Since $I_v$ is finite and $G_v$ is residually-$\mathcal{C}$ for every $v$ by assumption there is $F_v \in \mathcal{C}$ and a homomorphism $\phi_v \colon G_v \to F_v$ such that $\phi_v(g_i) \neq 1$ in $F_v$ for all $i \in I_v$. By Remark \ref{functorial_property_of_graph_products} we have the corresponding unique extension $\phi \colon G \to F$, where $F = \Gamma \mathcal{F}$. Clearly, $(\phi_{v_1}(g), \dots, \phi_{v_n}(g_n))$ is a $\Gamma$-reduced expression for $\phi(g)$ therefore $|g| = |\phi(g)|$ and $\supp(g) = \supp(\phi(g))$.
	
	Suppose that $g$ is $\Gamma$-cyclically reduced in $G$. Obviously $\FL(g) = \FL(\phi(g))$, $\LL(g) = \LL(\phi(g))$ and $S(g) = S(\phi(g))$ and thus $(\FL(\phi(g)) \cap \LL(\phi(g))) \setminus S(\phi(g)) = (\FL(g)\cap\LL(g))\setminus S(g) = \emptyset$ by Lemma \ref{graph_product_cyclically_reduced_criterion} because $g$ is $\Gamma$-cyclically reduced and therefore $\phi(g)$ is $\Gamma$-cyclically reduced in $F$ again by Lemma \ref{graph_product_cyclically_reduced_criterion}.
\end{proof}
In fact we are can generalise the previous lemma to any finite number of given elements.
\begin{cor}
\label{GP_pomocna_2}
Let $f, g \in G$ be $\Gamma$-cyclically reduced in $G$ and assume that $f \neq g$. Then there is $\mathcal{F} = \{F_v | v\in V\Gamma\}$, a family of $\mathcal{C}$-groups indexed by $V\Gamma$, and a homomorphism $\phi_v \colon G_v \to F_v$ for every $v \in V\Gamma$ such that for the corresponding extension $\phi \colon G \to F$, where $F = \Gamma\mathcal{F}$, all of the following are true:
	\begin{itemize}
		\item[(i)] $|g| = |\phi(g)|$ and $\supp(g) = \supp(\phi(g))$,
		\item[(ii)] $|f| = |\phi(f)|$ and $\supp(f) = \supp(\phi(f))$, 
		\item[(iii)] $\phi(f), \phi(g)$ are $\Gamma$-cyclically reduced in $F$,
		\item[(iv)] $\phi(f) \neq \phi(g)$ in $F$.
	\end{itemize}
\end{cor}
\begin{proof}
	We use Lemma \ref{GP_pomocna} on $g$,$f$ and $gf^{-1}$ to obtain three corresponding families $\mathcal{F}^f$, $\mathcal{F}^g$ and $\mathcal{F}^{fg^{-1}}$. For every $v \in V\Gamma$ we set $K_v = \ker(\phi_{v}^f) \cap \ker(\phi_{v}^g) \cap \ker(\phi_{v}^{fg^{-1}})$ and define $\phi_v \colon G_v \to F_v$, where $F_v = G_v / K_v$. Clearly the family of $\mathcal{C}$-groups $\mathcal{F} = \{F_v | v\in V\Gamma\}$ together with homomorphisms $\phi_v \colon G_v \to F_v$ and the extension $\phi \colon G \to \Gamma \mathcal{F}$ has all the claimed properties.
\end{proof}

The proof of the following remark is left as a simple exercise for the reader.
\begin{remark}
\label{direct}
	Suppose that $\C$ is a class of finite groups satisfying (c1) and (c2). Then the class of $\CCS$ groups is closed under taking direct products.
\end{remark}
\begin{proof}[Proof of Theorem \ref{main_cs}]
	Let $g \in G$ be arbitrary and let $f \in G$ such that $f \not\sim_G g$. Note that the set of vertices $X = \supp(g) \cup \supp(f) \subseteq V\Gamma$ is finite and $\rho_X(f) = f \not\sim_{G_X} \rho_X(g) = g$, where $\rho_X \colon G \to G_X$ is the canonical retraction corresponding to the full subgroup $G_X \leq G$ given by the set $X \subseteq V\Gamma$. Hence without loss of generality we may assume that $|V\Gamma| < \infty$. Let $f_0, g_0 \in G$ be $\Gamma$-cyclically reduced elements of $G$ such that $f_0 \sim_G f$ and $g_0 \sim_G g$. Clearly $f_0 \not\sim_G g_0$.  
	
	By Lemma \ref{conjugacy_criterion_for_graph_products} we have three possibilities to consider:
	\begin{itemize}
		\item[(i)] $\supp(g_0) \neq \supp(f_0)$ or $|g_0| \neq |f_0|$,
		\item[(ii)] $p(f_0)$ is not a cyclic permutation of $p(g_0)$,
		\item[(iii)] $s(f_0) \not\in s(g_0)^{G_{S(g_0)}}$.
	\end{itemize}
	Assume that either $\supp(g_0) \neq \supp(f_0)$ or $|f_0| \neq |g_0|$. Then we can use Corollary \ref{GP_pomocna_2} to obtain a family of $\mathcal{C}$-groups $\mathcal{F} = \{F_v|v \in V\Gamma\}$ and a homomorphism $\phi_v \colon G_v \to F_v$ for every $v \in V\Gamma$ such that for the corresponding extension $\phi \colon G \to \Gamma \mathcal{F}$ we have either $\supp(\phi(f_0)) \neq \supp(\phi(g_0))$ or $|\phi(f_0)| \neq |\phi(g_0)|$ respectively. By Lemma \ref{conjugacy_criterion_for_graph_products} we see that $\phi(f_0) \not\sim_{\Gamma\mathcal{F}}\phi(g_0)$ and hence $\phi(f) \not\sim_{\Gamma\mathcal{F}} \phi(g)$. Note that $\Gamma \mathcal{F}$ is a finite graph product of groups belonging to the class $\C$ and thus by Corollary \ref{finite} we see that the group $\Gamma \mathcal{F}$ is $\CHCS$.
	
	Assume that $\supp(g_0) = \supp(f_0)$ and $|g_0| = |f_0|$. Suppose that $p(f_0)$ is not a cyclic permutation of $p(g_0)$. Let $\{p_1, \dots, p_m\} \subset G$ be the set of all cyclic permutations of $p(g_0)$ including $p(g_0)$. Then $p_i \neq p(f_0)$ for $i = 1, \dots, m$ and we can use Corollary \ref{GP_pomocna_2} for each pair $p(f_0), p_i$, where $1 \leq i \leq m$, to obtain a family of $\mathcal{C}$-groups $\mathcal{F}_i = \{F_v^i | v \in V\Gamma\}$ with homomorphisms $\phi_v^i \colon G_v \to F_v^i$ for all $v \in V\Gamma$. For every $v \in V\Gamma$ set $K_v = \bigcap_{i = 1}^m \ker(\phi_v^i)$ and denote $F_v = G_v / K_v$. Set $\mathcal{F} = \{F_v | v\in V\Gamma\}$ and let $\phi_v \colon G_v \to F_v$ be the natural projection corresponding to $v$. Let $\phi \colon G \to \Gamma \mathcal{F}$ be the natural extension. Note that $p(\phi(f_0)) = \phi(p(f_0))$ and $p(\phi(g_0)) = \phi(p(g_0))$. Clearly the set $C = \{\phi(p_1),\dots, \phi(p_m)\}$ is the set of all cyclic permutations of $p(\phi(g_0))$ and we see that $p(\phi(f_0))\not\in C$ and thus $p(\phi(f_0))$ is not a cyclic permutation of $p(\phi(g_0))$. By Lemma \ref{conjugacy_criterion_for_graph_products} we see that $\phi(f_0) \not\sim_{\Gamma\mathcal{F}} \phi(g_0)$ and thus $\phi(f) \not\sim_{\Gamma\mathcal{F}} \phi(g)$. Again, by Corollary \ref{finite} we see that the group $\Gamma \mathcal{F}$ is $\CHCS$.

	Now assume that $\supp(g_0) = \supp(f_0)$, $|g_0| = |f_0|$, $p(f_0)$ is a cyclic permutation of $p(g_0)$. Since $\supp(f_0) = \supp(g_0)$ we see that $S(g_0) = S(f_0)$. Denote $S = S(g_0)$ and assume that $s(f_0) \not\in s(g_0)^{G_{S}}$. Note that 
	\begin{displaymath}
		G_{S} = \prod_{v \in S}G_v
	\end{displaymath}
	is a direct product of $\C$-CS groups and thus it is a $\C$-CS group by Remark \ref{direct}. Consider the retraction $\rho_S \colon G \to G_S$. Clearly $\rho_S(f_0) = s(f_0)$ and $\rho_S(g_0) = s(g_0)$. Therefore $\rho_S(g_0) \not \in \rho_S(g_0)^{G_S}$ by assumption and consequently $\rho_S(f) \not\in \rho_S(g)^{G_S}$.
	
	In each of the cases we have constructed a homomorphism onto a $\mathcal{C}\text{-CS}$ group, such that the images of $f$ and $g$ were not conjugate. Then by Lemma \ref{closed_simplification} we see that the conjugacy class $g^G$ is $\C$-closed in $G$. As $g$ was arbitrary we see that $G$ is $\mathcal{C}\text{-CS}$.		
\end{proof}

\subsection{Infinite graph products of $\CHCS$ groups}
The idea of the proof of Theorem \ref{main_chcs_infinite} somewhat similar to the proof of Theorem \ref{main_cs}. In the proof of Theorem \ref{main_cs} we started with a possibly infinite graph $\Gamma$ and showed that we can always retract to a full subgroup $G_A$ given by a finite set of vertices $A \subseteq V\Gamma$ and thus we were able to use Corollary \ref{finite}. In the proof of Theorem \ref{main_chcs_infinite} we start with a graph product $\Gamma \mathcal{G}$, where $\Gamma$ is an infinite graph, we show that for every $g \in G$ we can construct a finite graph product $\Delta \mathcal{D}$ of $\CHCS$ groups and a homomorphism $\delta \colon \Gamma \mathcal{G} \to \Delta \mathcal{D}$ such that $C_{\Delta \mathcal{D}}(\delta(g)) = \delta(C_G(g))$.
\begin{proof}[Proof of Theorem \ref{main_chcs_infinite}.]
	Let $\Gamma$ be a graph and let $\mathcal{G} = \{G_v \mid v\in V\Gamma\}$ be a family of groups such that the group $G_v$ is $\CHCS$ for every $v \in V\Gamma$. Let $G = \Gamma \mathcal{G}$. By Theorem \ref{main_cs} we see that the group $G$ is $\CCS$ so we need to show that $G$ satisfies $\C$-CC. Clearly, $G$ satisfies $\C$-CC if and only if for every $g \in G$ the pair $(G, g)$ satisfies $\C\text{-CC}_G$. Let $g \in G$ and $K \in \NC(G)$ be arbitrary. Pick $g' \in G$ such that $g \sim_G g'$ and $g'$ is $\Gamma$-cyclically reduced. By Lemma \ref{cc_simplification} we see that the pair $(G,g)$ has $\C\text{-CC}_G$ if and only if the pair $(G, g')$ has $\C\text{-CC}_G$. Denote $A = \supp(g')$.
	
	Let $\varphi \colon G \to G/K$ be the natural projection. Define a family of groups $\mathcal{F} = \{F_v \mid v \in V\Gamma\}$, where $F_v = G_v$ if $v \in A$ and $F_v = \varphi(G_v)$ otherwise. For every $v \in V\Gamma$ we have a group homomorphism $\phi_v \colon G_v \to F_v$ where $\phi_v = \text{id}_{G_v}$ if $v \in A$ and $\phi_v = \varphi \restriction_{G_v}$ otherwise. By Remark \ref{functorial_property_of_graph_products} there is a unique group homomorphism $\phi \colon \Gamma\mathcal{G} \to \Gamma \mathcal{F}$ such that $\phi \restriction_{G_v} = \phi_v$ for every $v \in V\Gamma$. Denote $F = \Gamma \mathcal{F}$. Note that $\ker{\phi} = \lnorm \ker(\phi_v)\rnorm_{v \in V\Gamma}^{G}$. Clearly if $v \in A$ then $\ker(\phi_v) = \{1\}$ and if $v \in V\Gamma \setminus A$ then $\ker(\phi_v) = K \cap G_v$. It follows that $\ker(\phi) \leq K$ and hence there is a unique homomorphism $\overline{\phi} \colon F \to G/K$ such that $\varphi = \overline{\phi}\circ \phi$.
		
	Set $A' = V\Gamma \setminus A$. Define equivalence $\approx_1$ on $A'$ as follows: $u \approx_1 v$ if $\link(u) \cap A = \link(v) \cap A$. Define equivalence $\approx_2$ on $A'$ as follows: $u \approx_K v$ if $\overline{\phi}(F_u) = \overline{\phi}(F_v)$ in $G/K$. Now let $\approx$ be the equivalence relation on $A'$ obtained as intersection of $\approx_1$ and $\approx_2$, i.e. $u \approx v$ if $\link(u) \cap A = \link(v) \cap A$ and $\overline{\phi}(F_u) = \overline{\phi}(F_v)$ in $G/K$. Note that $|A' / \approx_2| \leq 2^{|A|}$ and $|A' / \approx_2| \leq 2^{|G/K|}$, therefore we see that $|A' /\approx| < \infty$.
	
	Define a graph $\Delta$ with vertex set $V \Delta = A \cup (A' / \approx$). Note that $V\Delta$ is finite. For $u,v \in A$ we set $\{u,v\} \in E\Delta$ if and only if $\{u,v\} \in E\Gamma$, for $u \in A$ and $[x]_\approx \in A'/\approx$ we set $\{u, [x]_\approx\} \in E\Delta$ if and only if there is $x_0 \in [x]_\approx$ such that $\{u,x\} \in E\Gamma$. Similarly for $[x]_\approx,[y]_\approx \in A'/\approx$ we set $\{[x]_\approx,[y]_\approx\} \in E\Delta$ if and only if there are $x_0 \in [x]_\approx$ and $y_0 \in [y]_\approx$ such that $\{x_0, y_0\} \in E\Gamma$. Note that the natural map from $V\Gamma$ to $V\Delta$ actually extends to a graph morphism from $\Gamma$ to $\Delta$.
	
	To every vertex in $v\in V\Delta$ we assign a vertex group $D_v$ in the following way: if $v \in A$ then $D_v = G_v$; if $v = [v_0]_\approx$ for some $v_0 \in A'$ then $D_v = \varphi(G_{v_0}) = \overline{\phi}(F_v)$. This leads to a family of groups $\mathcal{D} = \{D_v \mid v \in V\Delta\}$. For every $v\in V\Gamma$ we define a group homomorphism $\overline{\varphi}_v \colon F_v \to D_{x_v}$, where $x_v = v$ if $v \in A$ and $x_v = [v]_\approx$ otherwise. If $v \in A$ then $\overline{\varphi}_v = \text{id}_{G_v}$ and if $v \in V\Gamma \setminus A$ then $\overline{\varphi}_v = \overline{\phi}\restriction_{F_v}$. By a theorem of von Dyck (see \cite[footnote 2, page 346]{rotman}) the family of group homomorphisms $\{\overline{\varphi}_v \mid v \in V\Gamma\}$ extends to a  homomorphism $\overline{\varphi} \colon F \to D$, where $D = \Delta \mathcal{D}$ is the corresponding graph product.
	
Let $x, y \in F$ be arbitrary. It is obvious that if $\overline{\varphi}(x) = \overline{\varphi}(y)$ then $\overline{\phi}(x) = \overline{\phi}(y)$ and thus $\ker(\overline{\varphi}) \leq \ker(\overline{\phi}) = \phi(K)$. We see that there is unique homomorphism $\overline{\delta} \colon D \to G/K$ such that $\overline{\phi} = \overline{\delta} \circ \overline{\varphi}$. Denote $\delta = \overline{\delta}\circ \overline{\phi}$. The following commutative diagram illustrates the situation.

\[
 	\xymatrix{
 		G=\Gamma \mathcal{G} \ar[dr]^{\phi} \ar@/_/[ddr]_{\varphi} \ar@/^/[drr]^{\delta}\\
		& F = \Gamma \mathcal{F} \ar[r]^{\overline{\varphi}} \ar[d]^{\overline{\phi}} &
 		D = \Delta \mathcal{D} \ar[dl]^{\overline{\delta}} \\
 		&
 		G / K
 	}
\]
Clearly $\ker(\delta) = \phi^{-1}(\ker(\overline{\varphi})) \leq K$.

Now we need to show that $C_{\delta(G)}(\delta(g)) \subseteq \delta(C_G(g)K)$. One can easily check that $\PC_{\Gamma}(\langle g'\rangle) = G_A$ and therefore by Lemma \ref{centralisers_in_graph_products} we have $C_G(g') = C_{G_A}(g') G_{\link(A)}$. Denote $\delta_A = \delta \restriction_{G_A}$. From the construction of $\delta$ it is easy to see that $\delta_A \colon G_A \to D_A$ is an isomorphism. Let $P = \PC_{\Delta}(\delta(\langle g'\rangle))$. As $D_A$ is a full (and hence parabolic) subgroup of $D$ and $\delta(g') \in P$ we see that $P \leq D_A$ due to minimality of $P$. By \cite[Lemma 3.7]{yago} we see that $P$ is actually parabolic in $\Delta_A \mathcal{D}_A = D_A$. Let $P' = \delta_A(P)^{-1} \leq G_A \leq G$. From the construction of the map $\delta$ we see that $P'$ is parabolic in $G_A$ (and thus in $G$) and that $g' \in P'$. Since $\PC_{\Gamma}(\langle g' \rangle) = G_A$ and $g' \in P'$ we see that $G_A \leq P'$ and therefore $P' = G_A$. This means that $P = D_A$. We see that $\PC_{\Delta}(\langle \delta(g') \rangle) = \PC_{\Delta}(\delta(\langle g' \rangle)) = D_A$ and hence by Lemma \ref{centralisers_in_graph_products} we get that $C_D(\delta(g')) = C_{D_A}(\delta(g'))D_{\link(A)}$.

Again, since $\delta \restriction_{G_A}$ is an isomorphism we see that $\delta(C_{G_A}(g')) = C_{D_A}(\delta(g')) = C_{\delta(G_A)}(\delta(g'))$. From the construction of the equivalence $\approx$ we see that for every $v \in V\Gamma$ we have $[v]_\approx \in \link(A)$ in $\Delta$ if and only if $v \in \link(A)$ in $\Gamma$ and hence $\delta(G_{\link(A)}) = D_{\link(A)}$. We see that 
	\begin{displaymath}
		C_D(\delta(g')) = C_{D_A}(\delta(g'))D_{\link(A)} = \delta(C_{G_A}(g')G_{\link(A)}) = \delta(C_G(g')) \subseteq \delta(C_G(g')K).
	\end{displaymath}
		
For every $v \in V\Delta$ the group $D_v$ is either an infinite $\CHCS$ group or belongs to the class $\C$. By Theorem \ref{main_hcs} we see that the group $D$ is $\CHCS$ and hence $D$ satisfies $\C$-CC by Theorem \ref{cc_and_cs_is_hcs}. Consequently, the pair $(D, \delta(g'))$ satisfies $\C$-CC in $D$. By Lemma \ref{cc_simplification} we see that the pair $(G,g)$ satisfies $\C\text{-CC}_G$ for any $g \in G$ and therefore $G$ satisfies $\C$-CC. We have proved that $G$ is $\C$-CS and satisfies $\C$-CC, hence by Theorem \ref{cc_and_cs_is_hcs} we see that $G$ is $\CHCS$.
\end{proof}

\subsection{Some corollaries}
Applying Theorem \ref{main_chcs_infinite} to the most obvious types of extension closed varieties of finitely presented groups we immediately get that the class of HCS groups is closed under taking finite graph products, similarly for $p$-HCS and (finite solvable)-HCS.

We can also extend the results of Minasyan (see \cite{raags}) and Toinet (see \cite{toinet}) to infinitely generated right angled Artin groups.
\begin{cor}
	Infinitely generated RAAGS are HCS and $p$-HCS for every prime number $p$.
\end{cor}

In \cite[Theorem 1.2]{racgs} Caprace and Minasyan proved that finitely generated RACGs are CS. By applying Theorem \ref{main_chcs_infinite} to RACGs once in the context of the class of all finite groups and once in the context of all finite 2-groups we get following strengthening of the mentioned result.
\begin{cor}
	\label{racgs_hcs}
	Arbitrary (possibly infinitely generated) right angled Coxeter groups are HCS and 2-HCS.
\end{cor}
The statement of Corollary \ref{racgs_hcs} can be compared with the following example: the group $G = \text{FSym}(X)$ of finitary permutations of an infinite set $X$ is an infinitely generated Coxeter group, but it is not even residually finite. Clearly being right angled is a strong requirement.

As we mentioned in the introductory section virtually polycyclic groups are HCS, thus we can state the following corollary.
\begin{cor}
	Let $\Gamma$ be any graph and let $\mathcal{G} = \{G_v \mid v \in V\Gamma\}$ be a family of groups such that the group $G_v$ is virtually polycyclic for every $v \in V\Gamma$. Then the group $G = \Gamma \mathcal{G}$ is HCS.
\end{cor}